\documentclass[pdflatex,sn-mathphys-num]{sn-jnl}
\usepackage[T1]{fontenc}
\usepackage{amsmath,amssymb,amsthm,booktabs,microtype,graphicx,xcolor}
\usepackage{caption,subcaption}
\usepackage{float,placeins}
\usepackage{mathtools}

\DeclareCaptionFont{sncaption}{\figurecaptionfont}
\hypersetup{hidelinks,pdftitle={A short open space curve of unit width},
 pdfsubject={Exact analytic construction and computer-assisted width certificate}}
\theoremstyle{thmstyleone}
\newtheorem{theorem}{Theorem}
\newtheorem{lemma}[theorem]{Lemma}
\newtheorem{proposition}[theorem]{Proposition}
\newtheorem{conjecture}[theorem]{Conjecture}
\usepackage{amsthm}
\theoremstyle{remark}
\newtheorem{remark}{Remark}
\newcommand{\R}{\mathbb R}
\newcommand{\conv}{\operatorname{conv}}
\newcommand{\norm}[1]{\left\lVert#1\right\rVert}
\newcommand{\ip}[2]{\langle #1,#2\rangle}
\newcommand{\atan}{\operatorname{atan2}}
\newcommand{\rev}{\operatorname{rev}}
\newcommand{\xx}{\mathbf x}
\newcommand{\eH}{\mathbf e_H}
\newcommand{\vH}{\mathbf v_H}
\newcommand{\ez}{\mathbf e_z}
\newcommand{\eF}{\mathbf e_F}
\newcommand{\vF}{\mathbf v_F}
\newcommand{\wF}{\mathbf w_F}
\newcommand{\av}{\mathbf a}
\newcommand{\gv}{\mathbf g}
\newcommand{\Tv}{\mathbf T}
\newcommand{\Nv}{\mathbf N}
\newcommand{\nv}{\mathbf n}
\newcommand{\mv}{\mathbf m}
\begin{document}
\title[A short open space curve of unit width]{A short open space curve of unit width}
\author[1]{\fnm{Tinko Sebastian} \sur{Bartels}}
\author[2]{\fnm{Dennis Kevin} \sur{Choy}}
\author[1]{\fnm{Ruixiang} \sur{Wu}}
\author*[1]{\fnm{Tongxin} \sur{Li}}
\email{litongxin@cuhk.edu.cn}

\affil[1]{\orgname{The Chinese University of Hong Kong, Shenzhen},
  \orgaddress{\street{2001 Longxiang Road, Longgang District},
    \city{Shenzhen}, \postcode{518172}, \country{China}}}
\affil[2]{\orgname{Technische Universit\"at Berlin},
  \orgaddress{\street{Stra\ss{}e des 17. Juni 135},
    \city{Berlin}, \postcode{10623}, \country{Germany}}}

\abstract{In this work, we propose a curve to be a minimum of length among rectifiable, continuous curves of unit width in three-dimensional Euclidean space. We show that a previously conjectured minimizer of that problem is not of unit width, give the numerical derivation of a candidate family of unit width curves, prove that it admits a strict local minimum within that candidate family and that this minimum is of unit width and length $3.905124162243593\ldots$. The curve consists of eight straight segments, two cone geodesics, and three circular helices. Its parameters are specified by a uniquely isolated solution of the joining and stationarity equations. To establish unit width in every direction, we combine analytic estimates with interval arithmetic on a finite covering of the remaining directions. We conjecture that the curve is globally minimizing. We supplement this work with the programs used to identify the candidates, replicate the computer-assisted parts of our proof, as well as a formalization of our proof in Lean.}

\maketitle

\vspace{-18pt}
\begingroup\footnotesize
\noindent\textbf{Acknowledgements}\par
\noindent The partially analytical method for width computation by enumerating critical point candidates from the projections of arcs, the discovery that the width of Zalgaller's curve is below $1$ through direction sampling, and the original attempts to improve Ghomi's closed curve through replacing helix pieces with straight lines were all previously made in the second author's unpublished master thesis, "Space Curves Minimizing the Ratio of Length to Width", Oct 14, 2021, that looked into potential improvements of Ghomi's closed curve.
\par\noindent This work was supported by the 1+1+1 Joint Collaboration Fund of The Chinese University of Hong Kong, The Chinese University of Hong Kong, Shenzhen, and the Department of Science and Technology of Guangdong Province to Tongxin Li, Tinko Sebastian Bartels, and Ruixiang Wu under Grant No. 2025A0505000047.
\par\endgroup
\clearpage

\section{Introduction}
We will first introduce the minimization problem, its previously proposed solution and its gaps, then introduce our own contribution. 

The minimum width of a curve is the thickness of its narrowest containing slab, that is the minimum of distance between two parallel planes that can contain the curve between them over all normal directions.
We consider rectifiable curves in three-dimensional Euclidean space. Formally a curve
$\gamma:[0,1]\to\R^3$, width is defined as 
\[
 w(\gamma)=\min_{\norm \nv=1}
 \left(\max_s\ip \nv{\gamma(s)}-\min_s\ip \nv{\gamma(s)}\right),
 \qquad m_{\rm open}=\inf_{w(\gamma)>0}\frac{L(\gamma)}{w(\gamma)}.
\]
Here $L$ is geometric length.
The quotient is invariant under scaling, so determining $m_{\rm open}$
is equivalent to minimizing length subject to $w=1$. The width of a curve or a more general set is the width of its convex hull.

The two-dimensional analogue arose in the study of escape paths from an infinite
strip, which is an instance of a more general minimization problem posed by Richard Bellman in \cite{bellman1956minimizationproblem}, commonly dubbed the "Lost in a forest problem"~\cite{Finch01102004}. Zalgaller~\cite{ZalgallerPlanar} found the minimizing caliper-shaped
arc, consisting of line segments and circular arcs, with length
approximately $2.2783$ at unit width. Adhikari and Pitman~\cite{AdhikariPitman}
rediscovered this solution and gave another proof of its optimality. In three-dimensional space, Zalgaller~\cite{Zalgaller} proposed a curve
$L_3$ obtained from cylindrical geodesics and polygonal pieces, and
conjectured that it minimizes length at unit width. Its reported length
is approximately $3.9215$. Ghomi~\cite[Theorem~1.1]{Ghomi} proved a
universal lower bound
\[
 m_{\rm open}\ge \sqrt{\ell^2_\text{caliper}+3^2} \approx 3.767 .
\]
Here $\ell_\text{caliper}$ is the length of the planar caliper curve at unit width.
His argument combines Wienholtz's alternating-support theorem with the
planar width bound and an inequality for the lengths of orthogonal
projections. These results leave a gap between the lower bound and the
proposed spatial construction.

Related computational approaches to escape-path problems include Deng's
traveling-salesman formulations with neighborhoods~\cite{DengGeneral},
subsequent convergence results and Lean formalizations~\cite{DengVariations},
and a computational treatment of the planar strip problem~\cite{DengStrip}.
The spatial examples in those works concern sphere containment or inspection rather than the open minimum-width problem considered here. 
For polygonal curves which have polyhedral convex hulls, width computation, i.e. minimization over all supporting directions reduces to finitely many antipodal face-vertex and edge-edge configurations, as in Houle and Toussaint~\cite{HouleToussaint}. 
We use this reduction in the numerical search described in Section~\ref{sec:search}.

The proof of unit width for Zalgaller's curve $L_3$ in~\cite[\S14]{Zalgaller} is incorrect.
For the parameters printed in \S12, the entire continuous curve $L_3$ lies
in the slab
\[
 0\le y+z\le t+c,\qquad
 w(L_3)\le\frac{t+c}{\sqrt2}<0.991103<1,
\]
where, in his notation, $t=\frac{a-b}{\sqrt{4a^2-1}}$ and
$(a,b,c)=(0.761337,0.064738,0.794982)$.

Define $s(x,y,z)=y+z$, then vertex formulas give
\[
s(A)=s(D)=0,\qquad s(C)=c+t>1.4016,
\]
whereas, with $T(x,y,z)=(-x,-y,z)$,
\[
\bigl(s(M),s(N),s(B),s(T(M)),s(T(N))\bigr)
\approx(0.027,0.121,0.188,0.533,1.225).
\]
The central helix $h(\theta)$ projects to $(y,z)=(\sin\theta,\cos\theta)$,
where $|\theta|\le\alpha$ and
\[
\alpha=\arctan\,(t/c)-\arccos\!\bigl((t^2+c^2)^{-1/2}\bigr)
\in(0.647,0.648).
\]
Since $\cos\theta-\sin\theta>0$ on this interval,
\[
0<\cos\alpha-\sin\alpha
\le s(h(\theta))
\le\cos\alpha+\sin\alpha<1.401<c+t.
\]
Every other curve piece is a line segment of which we only need to consider the endpoints. Thus the minimum is attained at $A,D$ and the
maximum at $C$, giving directional width
\[
w_{(0,1,1)/\sqrt2}(L_3)
=\frac{c+t}{\sqrt2}\approx0.99110214<1.
\]

The supplementary program \path{supplement/zalgaller_diagonal.py} verifies these endpoint comparisons and the helix monotonicity conditions by outward interval arithmetic, enclosing the width in this direction strictly below $1$. Its \texttt{--rounding-box} option checks robustness to rounding of the printed parameters. Interval arithmetic is a classical approach to guarantee bounds despite rounding error in floating-point calculations.
Figure~\ref{fig:zalgaller-slab} visualizes the curve's projection into the $yz$-plane and slab.

\begin{figure}[tbp]
\centering
\begin{subfigure}[t]{0.5494\linewidth}
\includegraphics[width=\linewidth]{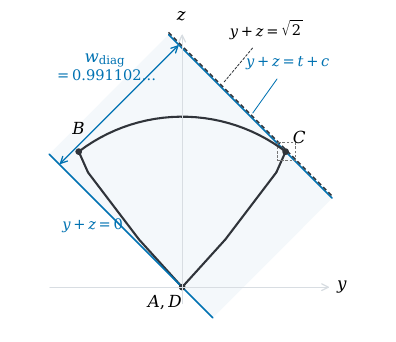}
\caption{Full $yz$ projection}
\end{subfigure}%
\begin{subfigure}[t]{0.4506\linewidth}
\includegraphics[width=\linewidth]{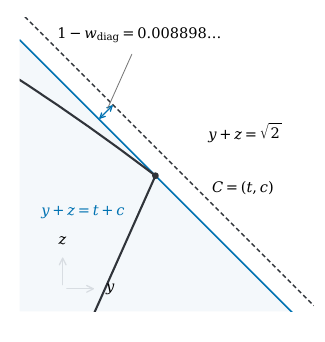}
\caption{Enlargement at $C$}
\end{subfigure}
\caption{Zalgaller's $L_3$ at the printed parameters.
In blue, we show projections of the supporting planes
$y+z=0$ and $y+z=t+c$ whose distance is
$w_{\rm diag}=(t+c)/\sqrt{2}=0.991102\ldots<1$.
The dashed line $y+z=\sqrt{2}$ is at distance one from the lower
plane and shows what plane the curve would have to reach to be of one width in that direction. Here $A,D$ project to $(0,0)$ and $C$ to $(t,c)$}
\label{fig:zalgaller-slab}
\end{figure}

\paragraph{Our contribution.} We construct a curve of width exactly one and length approximately
$3.905124162243593$. It partially follows Zalgaller's construction with regard to the arrangement of the endpoints and having a central arc following a cylinder. It differs by having outer arms formed from
tangent-cone geodesics, a pair of radius-$1/2$ helices, and straight
connectors. In Section~\ref{sec:search}, we present how a hypothesis for its construction was motivated by running a numerical optimization of unit width polygons for minimal length with iterative refinement through addition of vertices.
After obtaining a hypothesis in the form of a parametrized family of curves from tracing the shape of the best found polygons, we give the construction in Section~\ref{sec:curve}, isolate its exact parameters and prove local minimality of its length within that family in Section~\ref{sec:parameters}, and give a partially analytical, partially computer-assisted proof of its unit width in Section~\ref{sec:width}. 

Our main result is Theorem~\ref{thm:main}:

\begin{theorem}\label{thm:main}
The curve $\Gamma$ specified in Sections~\ref{sec:curve} and~\ref{sec:parameters}
has $w(\Gamma)=1$ and
\begin{equation}\label{eq:result}
\begin{split}
3.905124162243593023905759435469
 &< L(\Gamma)\\
 &<3.905124162243593023905759435470.
\end{split}
\end{equation}
It is a strict local minimum of $L/w$ within the three-parameter joining
family defined below. Consequently
$m_{\rm open}<3.905124162243594$.
\end{theorem}

Finally, in Section~\ref{sec:repro} we discuss supplementary material for reproducibility. We separately supply a formalized proof in Lean~4~\cite{Lean4} of our main-result that was generated using ChatGPT 6 Astra and verified to compile and match the main result statement up to slightly looser numerical bounds due different computation precision, namely $3.90512404<L(\Gamma)<3.90512429$. This upper bound is still below the reported length of $L_3$ in~\cite{Zalgaller}.
The sharper interval~\eqref{eq:result} can be reproduced using separately supplied Python programs. 
Figure~\ref{fig:overview} shows the curve in four views.

\section{Numerical Search for Candidate Curve}
\label{sec:search}
In this section we outline the search for the curve for which we give the exact construction, length, width and local optimality proof in the later sections. 
Because we give the full programs in supplementary material to allow replication by rerunning and the proofs are independent of how the candidate curve family was identified, we will keep this part brief.

Our strategy to obtain ideas for the structure of the minimal length curve of unit width was as follows. We first ran a numerical search for the minimizer among polygonal chains, then heuristically derived a hypothesis for an analytical construction from it, whose parameters were then again numerically optimized. 

We remember that every rectifiable curve can be approximated arbitrarily well through finite polygonal chains, i.e. chains of connected straight line segments.
We minimize the length to width ratio of polygonal chains over the vertex coordinates as optimization variables from several random starts, in case multiple local minima exist, and with increasingly refining vertex count.

The optimization problem is given as follows. For an ordered open polygon $P=(p_0,\ldots,p_{N-1})$, set
\[
 L(P)=\sum_{i=0}^{N-2}\norm{p_{i+1}-p_i},\qquad
 w(P)=\min_{\norm\nv=1}
 \bigl(\max_i\ip\nv{p_i}-\min_i\ip\nv{p_i}\bigr).
\]
The supplementary implementation \path{vertex_search/search.py} minimizes $L(P)/w(P)$
by sequential convex optimization. For each step, a convex
problem proposes a vertex displacement, and exact length and width bounds
determine whether to accept it. At the current polygon, it collects
support configurations with widths at most $(1+\varepsilon)w(P)$,
where $\varepsilon=0.002$. Let $w_j$ be the corresponding smooth feature
widths and $\mathcal I(P)$ be those with numerically usable derivatives. With vertex
displacements $\Delta=(\Delta_0,\ldots,\Delta_{N-1})$, solve
\begin{equation}\label{eq:polygon-step}
 \begin{aligned}
 \min_{\Delta}\quad&
   \sum_{i=0}^{N-2}\norm{p_{i+1}+\Delta_{i+1}-p_i-\Delta_i}
   +\frac{\eta}{2}\norm{\Delta}_2^2,\\
 \text{subject to}\quad&
   w_j(P)+D w_j(P)[\Delta]\ge w(P),\qquad j\in\mathcal I(P),\\
 &\norm{\Delta}_\infty\le\tau.
 \end{aligned}
\end{equation}
The quadratic penalty with coefficient $\eta=0.01$ discourages large steps, and the $\tau$ bounds each coordinate displacement.
Clarabel~\cite{Clarabel} solves an equivalent conic formulation with a quadratic objective. The Python implementation~\cite{Python} uses NumPy and
SciPy~\cite{NumPy,SciPy} for arrays, sparse matrices and linear algebra.

The linearized constraints may not preserve true width, so we recompute the width of each accepted curve exactly.
Coordinates are recentered and rounded to $2^{-32}\mathbb Z^3$. As we noted in the introduction, the width of a curve is the width of its convex hull. We use Qhull~\cite{Qhull} to obtain the triangulation of a convex hull and a fast C++ implementation to compute the convex hull's width using the approach of 
\cite{HouleToussaint}. By computing squared widths, we can guarantee exact results using integer arithmetic. Because the required precision requires big integer types whose operations are slow, we speed up the check with generated, semi-static
polynomial-sign filters from the implementation given in~\cite{Filters}. The filters certify polynomial signs used to identify antipodal features and compare squared widths using fast floating-point calculations. When the semi-static filter cannot make a guarantee, we fall back to exact big integer calculations. The implementation uses Boost.Geometry's vector operations
and Boost.Multiprecision's integer arithmetic~\cite{BoostGeometry,BoostMultiprecision}.

After optimization stops at a given vertex count, we refine the polygonal chain by adding additional vertices along existing segments. We add a slight displacement from the existing edges, so that newly added vertices' objective derivatives and constraints are not inactive by construction.

Independent instances can start from random Gaussian vertices or supplied
polygons. For half-turn or mirror symmetry, the implementation imposes
$p_{N-1-i}=Rp_i$, with $R=\operatorname{diag}(-1,-1,1)$ or
$R=\operatorname{diag}(-1,1,1)$, respectively. These runs use odd vertex
counts, add the corresponding linear equalities to the displacement problem,
and bisect edges in pairs. Coordinate rounding preserves the symmetry exactly.

The supplied population consists of 32 Gaussian starts: 12 unconstrained,
eight retaining half-turn symmetry, four retaining mirror symmetry, and four
releasing each symmetry. Seeds are $0$ through $31$. The 
initial vertex counts cycle through $9,13,17,25$ within each group. Vertex counts increase
by $N\mapsto2N-1$, capped at 3,999. The best run retained half-turn
symmetry and used seed $16$.
Its final polygon has the ratio interval bounds
\begin{equation}\label{eq:polygon-result}
 \frac{L(P)}{w(P)}\in
 [3.905124224044689335,\;3.905124224044689336].
\end{equation}
Twenty-six runs attained ratios in $[3.90512422,3.90512563]$:
all twelve unconstrained, six retaining half-turn symmetry, and all eight
releasing symmetry. After width
normalization, rigid alignment allowing reflections, and optional path
reversal, their root-mean-square distances to the best polygon, measured at
1,024 equally spaced arclength samples, were below $8.228\cdot10^{-6}$, which we interpreted as all competitive runs approximating the same local minimum. Seeds, compiler settings, stage vertices and results are in the supplementary material.

Subsequences of these chains are then inspected heuristically, to identify patterns that suggest approximation of elementary geometric shapes, e.g. near-collinearity of vertex sequences suggest straight line segments, approximately constant non-zero curvature and constant torsion suggest circular helices, and straightness after unfolding a fitted cone suggests cone geodesics.
Polygonal chains with additional symmetry constraints producing length over width ratios that remain competitive to unconstrained chains after refinement. This suggests that the minimizer has that symmetry.

After identifying an analytical candidate family, discussed in the next section, we implement a second length optimization over its parameters. For that, we replaced vertices by exact formulas for these pieces.
The unit separation of the parallel planes through $AC$ and $BD$ determines $t$
from $a,b$, and tangent joining leaves a local three-parameter family.
Using the eight-coordinate
formulation of Section~\ref{sec:parameters}, we minimized its analytic
length with SciPy's equality-constrained sequential least-squares programming (SLSQP) solver, then polished the joining and
stationarity equations at high precision with mpmath~\cite{Mpmath}.
The resulting endpoint parameters are
\[
 (a,b,c)\simeq(0.712497889,0.127319787,0.823130967).
\]

If this family contains a curve that is a local optimum approximated by the polygonal chain, we expect its length to be slightly lower than the best polygonal chain candidate due to the approximation error of polygonal chains. This is what we found in our testing. The selected curve's length is $3.905124162243593\ldots$, about
$6.18\cdot10^{-8}$ below the ratio of the best polygon in~\eqref{eq:polygon-result} and outside its interval bounds. Sections~\ref{sec:parameters} and~\ref{sec:width}
respectively isolate an exact stationary parameter tuple and certify
every supporting direction.

\section{Exact construction and length}\label{sec:curve}
Because the necessary number of parameters and notation is considerable, we will start this section with a full parametrization of the curve with approximate parameters to give an overview and reference that we can refer back to from every following section together with Figure~\ref{fig:overview} that has every joining point labeled and gives different projections.

\begin{figure}[H]
\centering
\begin{subfigure}[t]{0.5\linewidth}
\includegraphics[width=\linewidth]{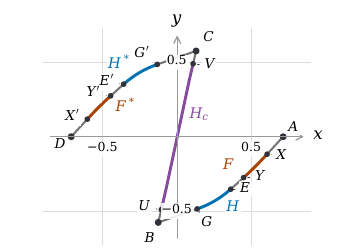}
\caption{$xy$ projection}
\end{subfigure}%
\begin{subfigure}[t]{0.5\linewidth}
\includegraphics[width=\linewidth]{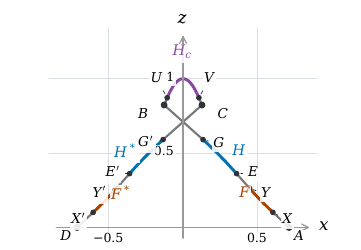}
\caption{$xz$ projection}
\end{subfigure}
\par\vspace{12pt}
\begin{subfigure}[t]{0.5\linewidth}
\includegraphics[width=\linewidth]{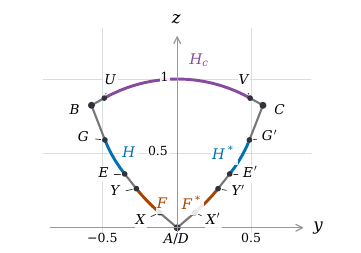}
\caption{$yz$ projection}
\end{subfigure}%
\begin{subfigure}[t]{0.5\linewidth}
\includegraphics[width=\linewidth]{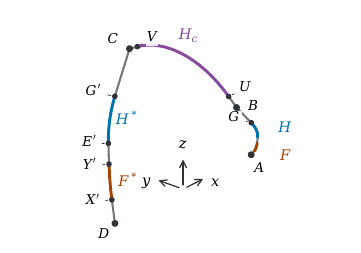}
\caption{Orthographic view}
\end{subfigure}
\par\vspace{2pt}
\includegraphics[width=\linewidth]{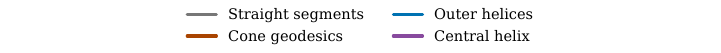}
\caption{Four views of the same thirteen-piece curve. All panels use the same
length scale and orthogonal projection. The label $A/D$ marks the coincident
free endpoints in the $yz$ projection. The labels
$X,Y,E$ are omitted from the orthographic view to avoid crowding.
Colors distinguish the eight
straight pieces, two cone geodesics, two outer helices and central helix.
Here $H_c$ denotes the central helix. Primes on points and stars on arcs
denote half-turn images}
\label{fig:overview}
\end{figure}

We give the notation for the following in Table~\ref{tab:notation}.
\begin{table}[!htbp]
\centering
\caption{Notation for the curve and its certificate}\label{tab:notation}
\small
\renewcommand{\arraystretch}{1.15}
\begin{tabular}{@{}p{0.37\linewidth}p{0.59\linewidth}@{}}
\toprule
Notation & Meaning\\
\midrule
Plain capitals & Named points in $\R^3$: $A,D$ are endpoints, $B,C$ are corners, and $X,Y,E,G,U,V$ are joins.\\
$R,\ P',\ F^*$ & Half-turn $R(x,y,z)=(-x,-y,z)$, images $P'=R(P)$ and $F^*=R(F)$.\\
\midrule
$F,H,H_c$ & Cone geodesic, outer helix, central helix.\\
$\mathcal A,\mathcal C,\Gamma$ & One outer arm, the central portion, the complete curve.\\
$\eH,\vH,\ez$ & Axis, horizontal (perpendicular to the axis) and vertical unit vectors for $H,H^*$.\\
$\eF,\vF,\wF$ & Axis, radial and circumferential unit vectors for the cone carrying $F$, at entry.\\
$\Tv_A,\Tv_\pm,\Tv_c,\bar\Tv,\ \Nv,\nv,\mv$ & Unit tangents, cone normal, general support normal, and unit slab normal.\\
$\gv,\av_\theta,\mathbf Q$ & Cone generator, azimuthal tangent and force vectors.\\
$\mathcal Q$ & Scalar-valued bilinear form defining the cone.\\
\bottomrule
\end{tabular}
\end{table}
\FloatBarrier

The curve, in order from $A$ to $D$, is
\begin{equation}\label{eq:wholecurve}
 \boxed{\quad
 \begin{gathered}
 \Gamma=\mathcal A*\mathcal C*R(\rev\mathcal A),\\
 \mathcal A=[A,X]*F*[Y,E]*H*[G,B],\qquad
 \mathcal C=[B,U]*H_c*[V,C].
 \end{gathered}\quad}
\end{equation}
Here $[P,Q]$ denotes the oriented segment from $P$ to $Q$,
$*$ concatenation, and $\rev$ reversal of paths.
Take five parameters $(a,b,c,\alpha,u)$, to be selected later in
Section~\ref{sec:parameters}. The points have coordinates
\begin{equation}\label{eq:endpoints}
\begin{gathered}
 \sigma=\sqrt{4a^2-1},\qquad t=(a-b)/\sigma,\qquad \ez=(0,0,1),\\
 A=(a,0,0),\quad B=(-b,-t,c),\quad C=(b,t,c),\quad D=(-a,0,0),\\
 X=D+r_0\gv_0,\qquad Y=E-\lambda\Tv_-,\\
 E=h\alpha\eH-\tfrac12\cos\alpha\,\vH
       +(z_0-\tfrac12\sin\alpha)\ez,\\
 G=-h\alpha\eH-\tfrac12\cos\alpha\,\vH
       +(z_0+\tfrac12\sin\alpha)\ez,\\
 U=(-k_c\beta,-\sin\beta,\cos\beta),\qquad V=R(U),\\
 R(x,y,z)=(-x,-y,z).
\end{gathered}
\end{equation}
The helix quantities used in these coordinates are
\begin{subequations}\label{eq:plot-auxiliaries}
\begin{equation}\label{eq:reconstruction}
\begin{gathered}
 \kappa=\tfrac12(\cos\alpha-u\sin\alpha),\qquad
 \rho=\sqrt{b^2+t^2-\kappa^2},\\
 h=\frac{\rho}{\alpha+u},\qquad
 z_0=c-\tfrac12(\sin\alpha+u\cos\alpha),\\
 \eH=\frac{(\rho b+\kappa t,\,\rho t-\kappa b,\,0)}{b^2+t^2},
 \qquad \vH=(-e_{H,y},e_{H,x},0),
\end{gathered}
\end{equation}
\begin{equation}\label{eq:helix-tangents}
 q_H=\sqrt{h^2+1/4},\qquad
 \Tv_\pm=\frac{-h\eH\pm\tfrac12\sin\alpha\,\vH+\tfrac12\cos\alpha\,\ez}{q_H},
\end{equation}
\begin{equation}\label{eq:central-data}
 p=\sqrt{t^2+c^2-1},\qquad
 \beta=\atan(t,c)-\arctan p,\qquad k_c=\frac b{\beta+p}.
\end{equation}
\end{subequations}
Here $\overrightarrow{PQ}=Q-P$, all angles are in radians, and
$\atan(y,x)\in(-\pi,\pi]$ is the two-argument arctangent.
Bars denote development (unrolling) into the plane.
The three curved pieces are the following functions of one variable:
\begin{equation}\label{eq:coneparam}
\begin{gathered}
 \bar Z(s)=((1-s)r_0+sr_1\cos(\mu\Theta),\;sr_1\sin(\mu\Theta)),\\
 F(s)=D+\norm{\bar Z(s)}\,
 \gv\!\left(\frac{\atan(\bar Z_y(s),\bar Z_x(s))}{\mu}\right),
 \qquad 0\le s\le1,
\end{gathered}
\end{equation}
\begin{equation}\label{eq:outer}
 H(\theta)=z_0\ez-h\theta\eH-\tfrac12\cos\theta\,\vH
                  +\tfrac12\sin\theta\,\ez,
 \qquad -\alpha\le\theta\le\alpha,
\end{equation}
\begin{equation}\label{eq:central}
 H_c(\theta)=(k_c\theta,\sin\theta,\cos\theta),
 \qquad -\beta\le\theta\le\beta.
\end{equation}
The cone quantities used in~\eqref{eq:endpoints} and~\eqref{eq:coneparam}
are grouped below by their role: frame, generators, tangency, exit data,
and development (unrolling) into the plane.
\begin{subequations}\label{eq:cone-auxiliaries}
\begin{equation}\label{eq:cone_frame}
\begin{gathered}
 d=\sqrt{\lVert\overrightarrow{DC}\rVert^2-1},\qquad
 \mu=\frac1{\lVert\overrightarrow{DC}\rVert},\qquad
 \nu=\frac d{\lVert\overrightarrow{DC}\rVert},\\
 \Nv_0=\left(0,\frac{t+pc}{t^2+c^2},\frac{c-pt}{t^2+c^2}\right),
 \qquad \gv_0=\frac{\overrightarrow{DC}-\Nv_0}{d},\\
 \eF=\frac{\overrightarrow{DC}}{\lVert\overrightarrow{DC}\rVert},\qquad
 \vF=\frac{\gv_0-\nu\eF}{\mu},\qquad \wF=\eF\times\vF,
\end{gathered}
\end{equation}
\begin{equation}\label{eq:cone-generators}
 \gv(\theta)=\nu\eF+\mu(\cos\theta\,\vF+\sin\theta\,\wF),\qquad
 \av_\theta=-\sin\theta\,\vF+\cos\theta\,\wF,
\end{equation}
\begin{equation}\label{eq:cone-tangency-data}
 \mathcal Q(\mathbf r,\mathbf s)
 =\ip{\mathbf r}{\overrightarrow{DC}}\ip{\mathbf s}{\overrightarrow{DC}}
       -d^2\ip{\mathbf r}{\mathbf s},\qquad
 \lambda=\frac{\mathcal Q(\overrightarrow{DE},\Tv_-)}{\mathcal Q(\Tv_-,\Tv_-)},
\end{equation}
\begin{equation}\label{eq:Y-r1-Theta-definition}
 r_1=\lVert\overrightarrow{DY}\rVert,\qquad
 \Theta=\atan(\ip{\overrightarrow{DY}}{\wF},\ip{\overrightarrow{DY}}{\vF}),
\end{equation}
\begin{equation}\label{eq:development}
 \bar A=(\ip{A-D}{\gv_0},\ip{A-D}{\wF}),\qquad
 \bar Y=r_1(\cos(\mu\Theta),\sin(\mu\Theta)),
\end{equation}
\begin{equation}\label{eq:developed-tangent}
 \bar\Tv=
 \begin{pmatrix}\cos(\mu\Theta)&-\sin(\mu\Theta)\\
                 \sin(\mu\Theta)&\cos(\mu\Theta)\end{pmatrix}
 \begin{pmatrix}\ip{\Tv_-}{\gv(\Theta)}\\\ip{\Tv_-}{\av_\Theta}\end{pmatrix},
\end{equation}
\begin{equation}\label{eq:cone-entry}
 r_0=\bar A_x-\bar A_y\frac{\bar T_x}{\bar T_y},\qquad \bar X=(r_0,0).
\end{equation}
\end{subequations}
Compute $E$ from~\eqref{eq:endpoints} before $\lambda$, then $Y$ before
$r_1,\Theta$, and finally $X$ once $r_0$ is known.
Here $F$ runs from $X$ to $Y$, $H$ from $E$ to $G$, and $H_c$ from
$U$ to $V$. The selected parameters are approximately
\[
 \begin{gathered}
 (a,b,c)\simeq(0.7125,\;0.1273,\;0.8231),\quad
 (\alpha,u)\simeq(0.2308,\;0.4784).
 \end{gathered}
\]
The reasoning for these choices and relationships follows from continuity and tangency constraints discussed in the remainder of this section, meaning that parametrized pieces must meet segments with matching tangent directions at the respective connection points, and the final parameter choices follow from analytical and numerical length minimization arguments to be given in  Section~\ref{sec:parameters}.
The curve has thirteen pieces, with $C^1$ joins except at $B$ and $C$.
Writing $P'=R(P)$ for half-turn images, the join-point order is
\[
 A,\ X,\ Y,\ E,\ G,\ B,\ U,\ V,\ C,\ G',\ E',\ Y',\ X',\ D.
\]
Figure~\ref{fig:overview} identifies these pieces.

\subsection{Endpoints and the outer helix}
Since $\sigma t=a-b$, the points $A,C$ lie in $x+\sigma y=a$ and
$B,D$ lie in $x+\sigma y=-a$. These planes have distance 1:
\begin{equation}\label{eq:a-sigma}
 \frac{2a}{\sqrt{1+\sigma^2}}=\frac{2a}{\sqrt{4a^2}}=1,
 \qquad a>\tfrac12.
\end{equation}
We will later show in Section~\ref{sec:width} that the curve is contained in this slab, which implies that its width is at most $1$. 

The outer helix pitch, height and frame are given in~\eqref{eq:reconstruction}.
On the selected branch, $\rho>0$ and $\alpha+u>0$. The identities
\[
 \norm{\eH}^2=\frac{\rho^2+\kappa^2}{b^2+t^2}=1,
 \qquad \rho\eH+\kappa\vH=(b,t,0)
\]
give an orthonormal frame $(\eH,\vH,\ez)$.
The vectors $\Tv_-,\Tv_+$ in~\eqref{eq:helix-tangents} are the unit
tangents at $E,G$, respectively.
The helix has radius $1/2$ around the line $z_0\ez+\R\eH$ and is shown in Figure~\ref{fig:outer-cylinder}. The other
outer helix is its half-turn image. By substitution, we get the tangent join to $B$:
\begin{equation}\label{eq:joinB}
 G+uH'(\alpha)=-\rho\eH-\kappa\vH+c\ez=B.
\end{equation}

\begin{figure}[H]
\centering
\begin{subfigure}[t]{0.5\linewidth}
\includegraphics[width=\linewidth]{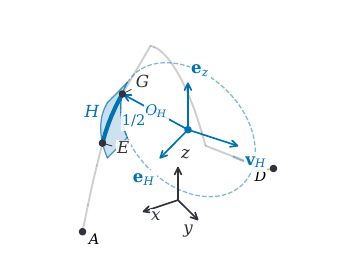}
\caption{Frame and radius}
\end{subfigure}%
\begin{subfigure}[t]{0.5\linewidth}
\includegraphics[width=\linewidth]{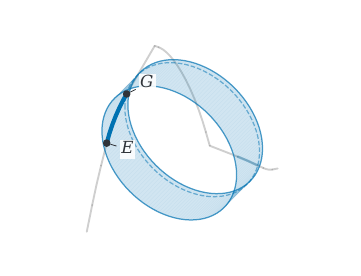}
\caption{Curve and supporting surface}
\end{subfigure}
\caption{The cylinder supporting $H$, from $E$ to $G$, in orthographic view.
Its surface is translucent and $H$ is highlighted. $H^*$, on the reverse arm of the curve, is faintly
colored near the label for $\vH$ and the rest of the curve is gray.
(a) The radius $[O_H,G]$ has length $1/2$, where
$O_H=z_0\ez-h\alpha\eH$. The colored arrows show
$(\eH,\vH,\ez)$ at $O_H$. Black coordinate arrows start at $(0,0,0)$.
Frame arrows are scaled for visibility.
In (b), we show a full cylinder piece for reference and only label the vertices of the curve, $E,G$, where the curve enters and leaves the cylinder}
\label{fig:outer-cylinder}
\end{figure}

\subsection{The central helix}
The central quantities $p,\beta,k_c$ are given in~\eqref{eq:central-data}, see Figure~\ref{fig:central-cylinder}.
In the $yz$-plane, $p$ is the tangent length from $(t,c)$ to the unit
circle. To check the joins, $t^2+c^2=1+p^2$ and
$\atan(t,c)=\beta+\arctan p$ give
\begin{equation}\label{eq:central-trig}
 t=\sin\beta+p\cos\beta,\qquad
 c=\cos\beta-p\sin\beta,\qquad b=k_c(\beta+p).
\end{equation}
Therefore
\[
 C-H_c(\beta)
 =(b-k_c\beta,\,t-\sin\beta,\,c-\cos\beta)
 =p(k_c,\cos\beta,-\sin\beta)=pH_c'(\beta).
\]
Similarly $H_c(-\beta)-B=pH_c'(-\beta)$, so both joins are tangent.

\begin{figure}[H]
\centering
\begin{subfigure}[t]{0.5\linewidth}
\includegraphics[width=\linewidth]{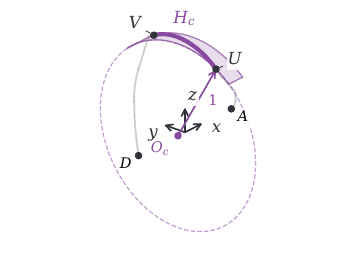}
\caption{Frame and radius}
\end{subfigure}%
\begin{subfigure}[t]{0.5\linewidth}
\includegraphics[width=\linewidth]{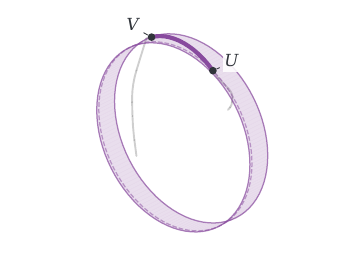}
\caption{Curve and supporting surface}
\end{subfigure}
\caption{The cylinder supporting $H_c$, from $U$ to $V$, in
orthographic view. Its surface is translucent and the curve piece $H_c$ is
highlighted. The rest of the curve is gray.
(a) The radius $[O_c,U]$ has length one, where $O_c=(-k_c\beta,0,0)$.
Black coordinate arrows start at $(0,0,0)$ and show the Cartesian frame,
with the $x$-direction along the cylinder axis.
In (b), we show a full cylinder segment for reference and only label the points where the curve enters and leaves the cylinder, $U,V$}
\label{fig:central-cylinder}
\end{figure}

\subsection{The cone and its development}
Next, we will construct the cone and its geodesics, shown in Figure~\ref{fig:cone}.
The frame and generators are given in~\eqref{eq:cone_frame}
and~\eqref{eq:cone-generators}.
Here $d$ is the tangent length from $D$ to the unit sphere centered at $C$.
$\mu,\nu$ are the sine and cosine of the cone's half-angle.
The vectors $(\eF,\vF,\wF)$ form an orthonormal frame. The circular cone $D+r \gv(\theta)$ has
unit inward normal
$\Nv(\theta)=\mu \eF-\nu(\cos\theta\,\vF+\sin\theta\,\wF)$.
Since $\mu^2+\nu^2=1$ and $\overrightarrow{DC}=\lVert\overrightarrow{DC}\rVert\eF$, it satisfies
\[
 \norm{\Nv(\theta)}^2=1,\qquad
 \ip{\Nv(\theta)}{\overrightarrow{DC}}=\mu\lVert\overrightarrow{DC}\rVert=1,\qquad
 \ip{\Nv(\theta)}{\gv(\theta)}=\mu\nu-\nu\mu=0.
\]
Thus every cone tangent plane passes through $D$ and is tangent to the
unit sphere centered at $C$.
Its entry tangent plane contains $AD$, since $(\Nv_0)_x=0$.

The quantities $\mathcal Q$ and $\lambda$ in~\eqref{eq:cone-tangency-data}
reduce cone tangency and unfolding to scalar operations.
The first cone-joining condition is
\begin{equation}\label{eq:conetangent}
 \mathcal Q(\overrightarrow{DE},\Tv_-)^2
 -\mathcal Q(\Tv_-,\Tv_-)\mathcal Q(\overrightarrow{DE},\overrightarrow{DE})=0.
\end{equation}
The cone is given by $\mathcal Q(Z-D,Z-D)=0$ on its positive
sheet, and intersection with the line $E-s\Tv_-$ gives the quadratic
\begin{align*}
 \mathcal Q(\overrightarrow{DE}-s\Tv_-,\overrightarrow{DE}-s\Tv_-)
 &=\mathcal Q(\overrightarrow{DE},\overrightarrow{DE})-2s\mathcal Q(\overrightarrow{DE},\Tv_-)\\
 &\quad+s^2\mathcal Q(\Tv_-,\Tv_-).
\end{align*}
Equation~\eqref{eq:conetangent} makes its discriminant zero. Its double
root is $s=\lambda$, since the quadratic coefficient is
nonzero on the selected branch.
The exit data $r_1,\Theta$ are given in~\eqref{eq:Y-r1-Theta-definition}.
In the selected branch $\Theta<0$, $\ip{\overrightarrow{DY}}{\eF}>0$ and
$Y=D+r_1\gv(\Theta)$.
Since $\gv'(\theta)=\mu\av_\theta$, with
$\norm{\gv}=\norm{\av_\theta}=1$ and $\ip{\gv}{\av_\theta}=0$,
\[
 \norm{\gv(\theta)\,dr+r\gv'(\theta)\,d\theta}^2
 =dr^2+\mu^2r^2d\theta^2.
\]
Thus unfolding preserves radial distances and replaces the azimuth
$\theta$ by $\mu\theta$. Place the entry generator, i.e. the ray from the cone peak $D$ through $X$, along the positive
horizontal axis and flatten its tangent plane in the basis
$(\gv_0,\wF)$. This gives the developed points in~\eqref{eq:development}.
Here $(\gv(\Theta),\av_\Theta)$ is an orthonormal basis of the tangent
plane at $Y$. The rotation in~\eqref{eq:developed-tangent} expresses $\Tv_-$
in the developed plane. The point $\bar X$ is where the line
$\bar A+s\bar\Tv$ meets the horizontal axis. Its second coordinate
vanishes at $s=-\bar A_y/\bar T_y$, giving the formula for $r_0$ in~\eqref{eq:cone-entry}.
The second cone-joining condition is
\begin{equation}\label{eq:unfoldjoin}
 \det(\bar Y-\bar A,\bar \Tv)=0.
\end{equation}
It makes the developed entry segment and cone path lie on the same line
as the developed outgoing tangent. The segment
$\bar Z(s)=(1-s)\bar X+s\bar Y$ develops back to the cone geodesic
$F$ in~\eqref{eq:coneparam}, with $X=F(0)$ and $Y=F(1)$.
Nonzero denominators, positive radii and tangent orientations are checked
on the box~\eqref{eq:box} by interval arithmetic in
\path{supplement/validate_symbolic.py}. \path{supplement/check_paper.py}
checks that this box lies inside the program's validated box.
All Lean files cited below are in the companion project and are checked
by its driver \texttt{verify\_all.py}.

\begin{figure}[H]
\centering
\begin{subfigure}[t]{0.5\linewidth}
\includegraphics[width=\linewidth]{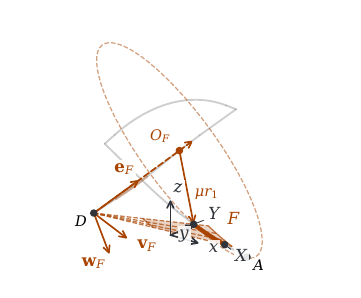}
\caption{Frame and radius}
\end{subfigure}%
\begin{subfigure}[t]{0.5\linewidth}
\includegraphics[width=\linewidth]{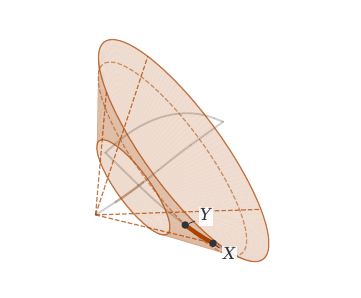}
\caption{Curve and supporting surface}
\end{subfigure}
\caption{The cone supporting $F$, from $X$ to $Y$, in orthographic view.
It has cone peak $D$ and is shown between $r=3/5$ and $r=7/5$. Panel (a)
is restricted to $\Theta\le\theta\le0$, the angular interval traversed by the highlighted curve piece $F$.
The half-turn image $F^*$ of $F$ is faintly colored near the label of $\eF$, better visible in (b), and the rest of the curve is gray.
In (a), the circular-section radius $[O_F,Y]$ has length $\mu r_1$, where
$O_F=D+r_1\nu\eF$. The dashed circle shows the full cross-section through $Y$.
The colored arrows show $(\eF,\vF,\wF)$ at the cone peak $D$.
black coordinate arrows start at $(0,0,0)$. Frame arrows are scaled for visibility.
In (b), we show a full cone slice for reference and label only the vertices where the curve enters and leaves $X,Y$}
\label{fig:cone}
\end{figure}

\begin{proposition}\label{prop:length}
On this joining branch the length is
\begin{equation}\label{eq:length}
 L=2\left(\norm{\bar Y-\bar A}+\lambda+(2\alpha+u)q_H\right)
       +2\sqrt{b^2+(\beta+p)^2}.
\end{equation}
\end{proposition}
\begin{proof}
Development preserves the lengths of the incoming segment and cone
geodesic. From the joining equations and the signs
$\bar A_y>0$, $\bar Y_y<0$, $\bar T_y<0$
(checked in \path{OpenCurve/RootGraph/DevelopedJoins.lean}) we get
\[
 \bar X-\bar A=-\frac{\bar A_y}{\bar T_y}\bar\Tv,
 \qquad
 \bar Y-\bar X=\frac{\bar Y_y}{\bar T_y}\bar\Tv,
 \qquad \norm{\bar\Tv}=1.
\]
Both multiples are positive, so $\bar X$ lies between $\bar A$ and
$\bar Y$, so their two lengths sum to $\norm{\bar Y-\bar A}$.
For the remaining pieces of the arm, we recall from~\eqref{eq:endpoints}
$E-Y=\lambda\Tv_-$, from~\eqref{eq:joinB} $B-G=uH'(\alpha)$ and
\[
 H'(\theta)=-h\eH+\tfrac12\sin\theta\,\vH
                   +\tfrac12\cos\theta\,\ez,
 \qquad \norm{H'(\theta)}=\sqrt{h^2+\tfrac14}=q_H.
\]
The respective lengths are therefore $\lambda$, $2\alpha q_H$ and
$uq_H$. Thus
\[
 L(\mathcal A)=\norm{\bar Y-\bar A}+\lambda+(2\alpha+u)q_H.
\]
The half-turn image $R(\rev\mathcal A)$ has the same length. For the central portion, including its two tangent leads,
$\norm{H_c'}=\sqrt{1+k_c^2}$ and each tangent lead has length
$p\sqrt{1+k_c^2}$ by~\eqref{eq:central-trig}. Hence
\begin{align*}
 L(\mathcal C)&=(2p+2\beta)\sqrt{1+k_c^2}\\
 &=2(\beta+p)\sqrt{1+\frac{b^2}{(\beta+p)^2}}\\
 &=2\sqrt{b^2+(\beta+p)^2}.
\end{align*}
Adding $2L(\mathcal A)+L(\mathcal C)$ proves~\eqref{eq:length}.
\end{proof}

\section{Derivation and isolation of the parameters}\label{sec:parameters}
For differentiation and certification within the parameter space of the family of curves described in the previous section, we introduce the augmented coordinates
\[
 \xx=(a,b,c,\phi,z_0,h,\alpha,u),\qquad
 \phi=\atan(e_{H,y},e_{H,x}).
\]
We treat $h,z_0,\phi$ as independent variables and set
$\eH=(\cos\phi,\sin\phi,0)$.
Recall that the joining condition $G+uH'(\alpha)=B$ from~\eqref{eq:joinB}
provides three scalar equations. Together with
\eqref{eq:conetangent} and~\eqref{eq:unfoldjoin},
these form the system $\mathcal J(\xx)=0$, from which~\eqref{eq:reconstruction} follows.
Indeed, the horizontal joining condition gives the following.
\[
 \begin{gathered}
 (b,t,0)=h(\alpha+u)\eH+\kappa\vH,\\
 h^2(\alpha+u)^2=b^2+t^2-\kappa^2=\rho^2.
 \end{gathered}
\]
Thus, we obtain $h(\alpha+u)=\rho$ on the positive branch. Solving for $(e_{H,x},e_{H,y})$
recovers the frame, and the vertical join gives $z_0$.
The box~\eqref{eq:box} fixes the branch of $\phi$.
Locally $\mathcal J=0$ determines $(\phi,z_0,h,\alpha,u)$ from $(a,b,c)$.
We can now derive three length-stationarity equations on this family that we can then solve for the three remaining free parameters.

\subsection{First variation}

We vary the parameters
\[
  \xx=(a,b,c,\phi,z_0,h,\alpha,u)
\]
within the box~\eqref{eq:box}, subject to the five joining equations
$\mathcal J(\xx)=0$. Near the selected solution, we treat $a,b,c$ as free parameters and choose $\phi,z_0,h,\alpha,u$ so that these equations remain satisfied.

Let $\xx(\varepsilon)$ be a smooth parameter path
that satisfies these conditions. For any scalar or
vector-valued quantity $F$, we define
\[
  \delta F
  :=\left.\frac{d}{d\varepsilon}
  F(\xx(\varepsilon))\right|_{\varepsilon=0}.
\]
Since the remaining parameters are locally determined
by $(a,b,c)$, the values $\delta a,\delta b,\delta c$
determine the infinitesimal variation.

These variations preserve the prescribed tangent joins.
At each smooth join, the endpoint contributions from
adjacent segments to the first variation of length
cancel because the segments share the same unit tangent.
Let $\Tv_A=(X-A)/\norm{X-A}$ and
$\Tv_c=(k_c,\cos\beta,\sin\beta)/\sqrt{1+k_c^2}$.
The straight development gives
\begin{gather*}
 \bar Y-\bar A=\norm{\bar Y-\bar A}\bar\Tv,\\
 j_F=\bar A_x\bar T_y-\bar A_y\bar T_x
     =r_0\bar T_y=r_1\ip{\Tv_-}{\av_\Theta}.
\end{gather*}
The determinant $j_F$ is constant along the developed straight line. 
Since $\bar\Tv$ is a unit vector, $\left|j_F\right|$ is its distance from the
developed peak. To differentiate the length, use
\begin{align*}
 \delta\norm{\bar Y-\bar A}&=\ip{\bar\Tv}{\delta\bar Y-\delta\bar A},\\
 \delta\bar Y&=(\cos(\mu\Theta),\sin(\mu\Theta))\,\delta r_1\\
       &\quad+r_1(-\sin(\mu\Theta),\cos(\mu\Theta))\,\delta(\mu\Theta).
\end{align*}
Using $\delta(\mu\Theta)=\Theta\delta\mu+\mu\delta\Theta$, we get
\begin{equation}\label{eq:developedvariation}
 \begin{aligned}
 \delta\norm{\bar Y-\bar A}
 &={}\ip{\Tv_-}{\gv(\Theta)}\delta r_1
  +j_F(\Theta\delta\mu+\mu\delta\Theta)\\
 &\quad-\bar T_x\delta\bar A_x-\bar T_y\delta\bar A_y.
 \end{aligned}
\end{equation}
To express it in the moving three-dimensional frame, we define
\begin{equation}\label{eq:cone_force}
 \mathbf Q=j_F\mu\left(\mu\Theta \eF-
 \nu\left(\sin\Theta\,\vF+(1-\cos\Theta)\wF\right)\right).
\end{equation}
An orthonormal frame has infinitesimal variations of the form
\[
 \delta\eF=\omega_1\vF+\omega_2\wF,\quad
 \delta\vF=-\omega_1\eF+\omega_3\wF,\quad
 \delta\wF=-\omega_2\eF-\omega_3\vF.
\]
Differentiate
$A-D=\bar A_x\gv_0+\bar A_y\wF$,
$Y-D=r_1\gv(\Theta)$ and $C-D=\eF/\mu$.
Using $\mu\delta\mu+\nu\delta\nu=0$ and
$\Tv_A=\bar T_x\gv_0+\bar T_y\wF$ gives the three scalar products
\begin{align*}
 \ip{\Tv_-}{\delta(Y-D)}
 ={}&\ip{\Tv_-}{\gv(\Theta)}\delta r_1+j_F\mu\delta\Theta\\
 &+j_F\{\nu(-\sin\Theta\,\omega_1+\cos\Theta\,\omega_2)
                    +\mu\omega_3\},\\
 \ip{\Tv_A}{\delta(A-D)}
 ={}&\bar T_x\delta\bar A_x+\bar T_y\delta\bar A_y
                    +j_F(\nu\omega_2+\mu\omega_3),\\
 \ip{\mathbf Q}{\delta(C-D)}
 ={}&-j_F\Theta\delta\mu
       -j_F\nu\{\sin\Theta\,\omega_1+(1-\cos\Theta)\omega_2\}.
\end{align*}
Subtracting the last two from the first cancels every $\omega_i$ term
and leaves~\eqref{eq:developedvariation}. Thus
\[
 \delta\norm{\bar Y-\bar A}
 =\ip{\Tv_-}{\delta(Y-D)}-\ip{\Tv_A}{\delta(A-D)}
  -\ip{\mathbf Q}{\delta(C-D)}.
\]
Adding the straight lead $[Y,E]$,
whose variation is $\ip{\Tv_-}{\delta E-\delta Y}$, gives
\begin{equation}\label{eq:variation}
 \begin{aligned}
 \delta\bigl(\norm{\bar Y-\bar A}+\lambda\bigr)
 &={}\ip{\Tv_-}{\delta E}-\ip{\Tv_A}{\delta A}\\
 &\quad-\ip{\Tv_--\Tv_A-\mathbf Q}{\delta D}-\ip{\mathbf Q}{\delta C}.
 \end{aligned}
\end{equation}
Here $-\ip{\mathbf Q}{\delta C}$ accounts for moving the center of the
unit sphere that determines the cone.

For the outer helix, recall $E=H(-\alpha)$, $G=H(\alpha)$ and
$\Tv_\pm=H'(\pm\alpha)/q_H$. Differentiating~\eqref{eq:outer}, with
$\partial_\phi\eH=\vH$ and $\partial_\phi\vH=-\eH$, shows that the
$\delta\phi$ and $\delta z_0$ terms in the two endpoint scalar products
agree. The remaining terms are
\[
 \ip{\Tv_-}{\delta E}-\ip{\Tv_+}{\delta G}
 =-2q_H\,\delta\alpha-\frac{2\alpha h}{q_H}\,\delta h,
 \qquad \delta q_H=\frac{h}{q_H}\,\delta h.
\]
They cancel the variation of $2\alpha q_H$, giving
\begin{equation}\label{eq:cancellation}
 \ip{\Tv_-}{\delta H(-\alpha)}-\ip{\Tv_+}{\delta H(\alpha)}
       +\delta(2\alpha q_H)=0.
\end{equation}
The final straight piece contributes
$\ip{\Tv_+}{\delta B}-\ip{\Tv_+}{\delta G}$. Combining it with
\eqref{eq:variation} and~\eqref{eq:cancellation} gives the variation of
one arm:
\[
 \delta L(\mathcal A)
 =\ip{\Tv_+}{\delta B}-\ip{\Tv_A}{\delta A}
  -\ip{\Tv_--\Tv_A-\mathbf Q}{\delta D}-\ip{\mathbf Q}{\delta C}.
\]
For the central portion, differentiate the tangent decomposition
in~\eqref{eq:central-trig}:
\[
 (t,c)=(\sin\beta,\cos\beta)+p(\cos\beta,-\sin\beta).
\]
It gives
\[
 (\delta t,\delta c)
 =(\cos\beta,-\sin\beta)\,\delta(\beta+p)
 -p(\sin\beta,\cos\beta)\,\delta\beta.
\]
Taking the scalar product with $(\cos\beta,-\sin\beta)$ eliminates
the last term, so
\[
 \delta(\beta+p)
 =\cos\beta\,\delta t-\sin\beta\,\delta c.
\]
Since $k_c=b/(\beta+p)$, differentiating the central length gives
\begin{align*}
 \tfrac12\delta L(\mathcal C)
 &=\frac{b\,\delta b+(\beta+p)\,\delta(\beta+p)}
         {\sqrt{b^2+(\beta+p)^2}}\\
 &=\frac{k_c\,\delta b+\cos\beta\,\delta t-\sin\beta\,\delta c}
         {\sqrt{1+k_c^2}}\\
 &=T_{c,x}\delta b+T_{c,y}\delta t-T_{c,z}\delta c.
\end{align*}
Finally, $L=2L(\mathcal A)+L(\mathcal C)$ and the endpoint coordinates give
\begin{align*}
 \delta A&=(\delta a,0,0),& \delta D&=(-\delta a,0,0),\\
 \delta B&=(-\delta b,-\delta t,\delta c),&
 \delta C&=(\delta b,\delta t,\delta c).
\end{align*}
Substituting these in the arm variation yields
\begin{align*}
 \tfrac12\delta L={}&(T_{-,x}-2T_{A,x}-Q_x)\delta a
 +(T_{c,x}-T_{+,x}-Q_x)\delta b\\
 &+(T_{c,y}-T_{+,y}-Q_y)\delta t
 +(T_{+,z}-T_{c,z}-Q_z)\delta c.
\end{align*}
To eliminate $\delta t$, recall $t=(a-b)/\sigma$ from~\eqref{eq:endpoints} and
$\sigma^2=4a^2-1$ from~\eqref{eq:endpoints}. Hence
\[
 t_a=\frac1\sigma-\frac{4a(a-b)}{\sigma^3}
     =\frac{4ab-1}{\sigma^3},\qquad
 t_b=-\frac1\sigma,\qquad \delta t=t_a\delta a+t_b\delta b.
\]
The three remaining equations are
\begin{equation}\label{eq:stationary}
 \boxed{\begin{aligned}
 f_1&=T_{-,x}-2T_{A,x}-Q_x
       +\frac{4ab-1}{\sigma^3}(T_{c,y}-T_{+,y}-Q_y)=0,\\
 f_2&=T_{c,x}-T_{+,x}-Q_x
       -\frac{T_{c,y}-T_{+,y}-Q_y}{\sigma}=0,\\
 f_3&=T_{+,z}-T_{c,z}-Q_z=0.
 \end{aligned}}
\end{equation}
Let $\mathbf y(a,b,c)=(\phi,z_0,h,\alpha,u)$ be the
local solution of the five joining equations. The \emph{reduced length}
is the full curve length expressed in the three independent parameters:
\[
 L_{\rm red}(a,b,c)=L(a,b,c,\mathbf y(a,b,c)).
\]
The calculation above says
\begin{align*}
 \delta L_{\rm red}&=2(f_1\delta a+f_2\delta b+f_3\delta c),\\
 \nabla L_{\rm red}&=2\mathbf f(a,b,c,\mathbf y(a,b,c)).
\end{align*}
Here $\mathbf f=(f_1,f_2,f_3)$, and the factor $2$ undoes the
normalization $\tfrac12\delta L$. Thus $\mathbf f=0$ expresses
stationarity under every variation within this joining family.

\subsection{An exact root and its certificate}
Let $\xx_*$ be the unique solution of $\mathcal J=0$, $\mathbf f=0$ in the rational box 

\begin{equation}\label{eq:box}
\begin{gathered}
\mathcal{B}\coloneqq \{ \xx:\norm{\xx - \hat\xx}_\infty\leq 10^{-10} \},\qquad  \hat \xx \coloneqq \left( a,b,c,\phi,z_0,h,\alpha,u \right)\\
 \begin{array}{c|r@{\qquad}c|r}
 a&0.7124978889260254&\phi&0.5323434101774462\\
 b&0.1273197867039577&z_0&0.4759161674352577\\
 c&0.8231309666744696&h&0.5672477790020362\\
 \alpha&0.2307732127386915&u&0.4783812610165119
 \end{array}.
 \end{gathered}
\end{equation}

Let
$\mathcal E=(\mathcal J,\mathbf f)$, the combined joining and stationarity equalities,
using $j_F=r_0\ip{\Tv_A}{\wF}$ to define $\mathbf f$ throughout $\mathcal B$.
Rational interval bounds verify every
denominator and angle branch. Interval arithmetic proves, for a rational
approximate inverse $\mathsf M$,
\[
 \sup_{\xx\in\mathcal B}
 \norm{I-\mathsf M\mathcal E'(\xx)}_\infty\le10^{-3},
 \qquad
 \norm{\mathsf M\mathcal E(\widehat{\xx})}_\infty\le10^{-15}.
\]
Let $\mathcal T(\xx)=\xx-\mathsf M\mathcal E(\xx)$.
This is a Newton step with the approximate inverse held fixed.
Its derivative is $I-\mathsf M\mathcal E'(\xx)$, so the first bound
makes the map $\mathcal T$ a contraction with constant $10^{-3}$ on the convex
box $\mathcal B$. It also makes $\mathsf M\mathcal E'(\xx)$ invertible
(its distance from $I$ is less than one), hence $\mathsf M$ is nonsingular.
For $\xx\in\mathcal B$,
\begin{align*}
 \norm{\mathcal T(\xx)-\widehat{\xx}}_\infty
 &\le\norm{\mathcal T(\xx)-\mathcal T(\widehat{\xx})}_\infty
       +\norm{\mathsf M\mathcal E(\widehat{\xx})}_\infty\\
 &\le10^{-3}\norm{\xx-\widehat{\xx}}_\infty+10^{-15}
 \le10^{-3}10^{-10}+10^{-15}<10^{-10}.
\end{align*}
Thus $\mathcal T$ maps $\mathcal B$ into itself. Its fixed points are
exactly the zeros of $\mathcal E$, since $\mathsf M$ is nonsingular.
From Banach's fixed-point theorem it follows that there exists a unique zero $\xx_* \in \mathcal B$. At this fixed
point the estimate yields
$\norm{\xx_*-\widehat{\xx}}_\infty
 \le10^{-3}\norm{\xx_*-\widehat{\xx}}_\infty+10^{-15}$, hence
\[
 \norm{\xx_*-\widehat{\xx}}_\infty
 \le\frac{10^{-15}}{1-10^{-3}}<10^{-12}.
\]
The root is therefore in the interior of $\mathcal B$. Through interval evaluation in
\path{OpenCurve/LengthBounds.lean} we get the bound
\begin{equation}\label{eq:leanlength}
 3.90512404<L(\Gamma)<3.90512429.
\end{equation}
For~\eqref{eq:result}, \path{supplement/validate_symbolic.py} isolates a root
within $4.721\,10^{-66}$ of its recorded rational center.
The program \path{supplement/check_paper.py} checks that this enclosure
lies in $\mathcal B$, so uniqueness identifies the same $\xx_*$, and
evaluates the sum-of-pieces form of~\eqref{eq:length} on this finer
enclosure to obtain the sharper bound.

\subsection{Strict local minimality in the joining family}
We partition $\xx=(\xx_{\rm f},\xx_{\rm d})$ into
free parameters $\xx_{\rm f}=(a,b,c)$ and dependent parameters
$\xx_{\rm d}=(\phi,z_0,h,\alpha,u)$. Write $D_{\rm f},D_{\rm d}$ for
the corresponding partial derivatives.
An interval bound on a rational approximate inverse, checked in
\path{OpenCurve/RootGraph/ReducedDerivative.lean}, verifies that $D_{\rm d}\mathcal J$ is invertible throughout $\mathcal B$.
All required expressions are smooth on a neighborhood of the root.
From the implicit function theorem it therefore follows that there is a local $C^2$ parametrization
$\xx_{\rm d}=\mathbf y(\xx_{\rm f})$ that we use above. Differentiating
$\mathcal J(\xx_{\rm f},\mathbf y(\xx_{\rm f}))=0$, we get
\[
 D_{\rm f}\mathcal J+D_{\rm d}\mathcal J\,D\mathbf y=0,
 \qquad D\mathbf y=-(D_{\rm d}\mathcal J)^{-1}D_{\rm f}\mathcal J.
\]
Differentiating $\nabla L_{\rm red}=2\mathbf f$ also accounts for the
adjustment of the dependent parameters needed to preserve the joins:
\begin{align*}
 \mathsf H=D^2L_{\rm red}
 &=2\left(D_{\rm f}\mathbf f+D_{\rm d}\mathbf f\,D\mathbf y\right)\\
 &=2\left(D_{\rm f}\mathbf f-
 D_{\rm d}\mathbf f\,(D_{\rm d}\mathcal J)^{-1}D_{\rm f}\mathcal J\right).
\end{align*}
Throughout the part of the joining family in $\mathcal B$, the same
file checks interval bounds showing that each Hessian entry differs by at most one from the
corresponding entry of the integer matrix
\begin{equation}\label{eq:hessian}
 \mathsf H_*=
 \begin{pmatrix}23&21&-15\\21&29&-22\\-15&-22&32\end{pmatrix},
 \qquad |H_{ij}-(H_*)_{ij}|\le1.
\end{equation}
The symmetric matrix $2\mathsf H_*-7I$ has leading principal minors
$39$, $225$ and $2301$, all positive. From Sylvester's criterion, we therefore get
for every $\mathbf v=(v_1,v_2,v_3)\in\R^3$,
\[
 \mathbf v^{\mathsf T}\mathsf H_*\mathbf v
 \ge\tfrac72\norm{\mathbf v}^2.
\]
To transfer this bound to $\mathsf H$, we use the entrywise error:
\[
 \left|\mathbf v^{\mathsf T}(\mathsf H-\mathsf H_*)\mathbf v\right|
 \le\sum_{i,j}|v_i||v_j|
 =\left(\sum_i|v_i|\right)^2\le3\norm{\mathbf v}^2.
\]
The last step is the Cauchy-Schwarz inequality. Hence
\[
 \mathbf v^{\mathsf T}\mathsf H\mathbf v
 \ge\tfrac12\norm{\mathbf v}^2.
\]
For a sufficiently small free-parameter displacement $\mathbf v$, the
parametrized family stays inside $\mathcal B$ along a segment from
$\xx_{{\rm f},*}$ to $\xx_{{\rm f},*}+\mathbf v$. Since
$\nabla L_{\rm red}(\xx_{{\rm f},*})=2\mathbf f(\xx_*)=0$, the Taylor
formula with integral remainder gives
\begin{align*}
 &L_{\rm red}(\xx_{{\rm f},*}+\mathbf v)-L_{\rm red}(\xx_{{\rm f},*})\\
 &\qquad=\int_0^1(1-s)\,
   \mathbf v^{\mathsf T}\mathsf H(\xx_{{\rm f},*}+s\mathbf v)\mathbf v\,ds\\
 &\qquad\ge\tfrac14\norm{\mathbf v}^2.
\end{align*}
This proves the strict local minimality of reduced length at $\xx_*$.
Section~\ref{sec:width} proves $w=1$ throughout the
nearby joining family, which therefore shows strict local optimality of the ratio $L/w$ within that curve family on $\mathcal B$.

\section{Width in every direction}\label{sec:width}
The proofs in this section frequently defer to computer-assisted bounds and dense computations, that can be verified individually through substitutions of the previously defined quantities and application of elementary algebraic identities but are not of particular interest to expand individually. They can be verified collectively through the supplementary Lean~4 proof. Instead we want to preface them with an intuitive overview of the proof idea and structure.

To prove unit width, we need to identify at least one direction $\mv$ for which the directional width, which is always an upper bound of the overall width, is exactly $1$, and show that the directional width is no lower than $1$ for every remaining direction. For every directional width there must be at least two support points, referred to as witnesses in the supplementary programs, one point per plane of the corresponding slab. For the curve to be optimal, we expect every point on a non-straight piece to be a support point for a direction whose directional width is $1$, otherwise the curve could be shortened by replacing part of the non-straight piece with a line segment, which would be shorter without losing any necessary support. Note, how Ghomi observes that this is not the case for his proposed closed curve in~\cite[Note~5.2]{Ghomi} and infers it is not optimal and can be shortened near the corners without losing width. Similar holds for end and corner points. We will therefore have to give analytical arguments for subsets of all directions corresponding to the non-straight pieces and end points for which directional width must be critical, in Sections~\ref{sec:slab} and \ref{sec:analytic-support}.

If we can establish directional width of at least $1$ for all critical directions and some open neighborhood of these direction subsets on the unit sphere, see the colored regions on Figure~\ref{fig:width-proof}, we need to establish the bound for the remaining compact complement on the unit sphere on which directional width is not critical. For any non-critical direction $\nv$ directional width of that direction's support points must have a lower bound greater than $1$ on some open neighborhood of $\nv$ on the sphere because directional width is continuous. This implies that an open cover of that compact remaining subset of directions exists, which must have a finite subcover, where each open direction set corresponds to two fixed curve points whose projected separation is strictly greater than $1$ on it. This can then be verified using computer assistance with sufficient numerical precision and interval arithmetic. An analogous argument, using a spherical triangulation rather than a cover of open cones, is made in Section~\ref{sec:finite-cover}. In that version of the argument a triangle vertex can also have critical width if that is analytically established.

Whenever we credit a bound to interval arithmetic, we mean outward-rounded interval evaluation using the parameter enclosures and procedures specified in the corresponding supplementary checker. The supplementary material~\cite{bartels_2026_23309323} contains programs for the replication of all cited bounds.
Its \path{INTERVAL_INDEX.txt} locates the verification of each interval bound in the manuscript.

For a compact set $\mathcal S$, write $h_{\mathcal S}(\nv)=\max_{P\in\mathcal S}\ip \nv P$.
Let $\mathcal K=\conv\Gamma$ and
$W(\nv)=h_{\mathcal K}(\nv)+h_{\mathcal K}(-\nv)=h_{\mathcal K-\mathcal K}(\nv)$.
Thus $W(\nv)/\norm\nv$ is the distance between the supporting parallel planes
normal to $\nv\ne0$ or the directional width along direction $\nv$. We prove $W(\nv)\ge\norm \nv$ in every direction,
with equality in one.
Every pair $P,Q\in\mathcal K$ gives the lower bound
$W(\nv)\ge\ip\nv{P-Q}$. The half-turn symmetry $R\Gamma=\Gamma$ gives
$W(R\nv)=W(\nv)$. Also
$W(-\nv)=W(\nv)$ and $W(s\nv)=sW(\nv)$ for $s\ge0$, so a bound for unit
normals extends to nonzero normals by multiplying by their length.
All bounds below hold for every $\xx\in\mathcal B$ satisfying
the five joining equations $\mathcal J=0$.
Figure~\ref{fig:width-proof} shows the analytic regions and the
subdivision used for the remaining unit normals.

\begin{figure}[H]
\centering
\includegraphics[width=\linewidth]{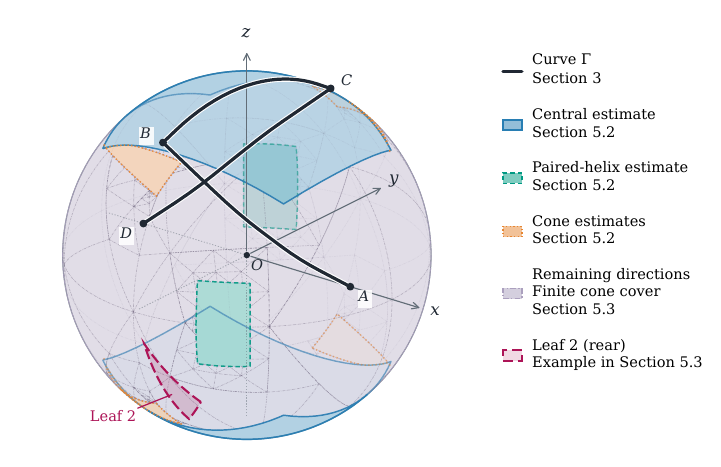}
\caption{Directions used in the width proof. The sphere represents the set of directions and our curve is drawn in the same coordinate system.
Colored patches with distinct border styles show the analytic
regions of Section~\ref{sec:analytic-support}, including their symmetry
images. Overlapping regions are drawn only once for readability. The labeled leaf 2 (dashed outline because it is on the backside of the sphere) is the spherical triangle used as the example in Section~\ref{sec:finite-cover}}
\label{fig:width-proof}
\end{figure}

\subsection{A unit containing slab}\label{sec:slab}
Recall $\sigma^2=4a^2-1$ and $\sigma t=a-b$ from~\eqref{eq:endpoints}. The vector
$\mv=(1,\sigma,0)/(2a)$ has
\[
 \norm\mv^2=\frac{1+\sigma^2}{4a^2}=1,\qquad
 \ip\mv C=\frac{b+\sigma t}{2a}=\frac12.
\]
The other endpoint projections follow in the same way:
\[
 \ip \mv A=\ip \mv C=\tfrac12,\qquad
 \ip \mv B=\ip \mv D=-\tfrac12.
\]
We check containment on whole arcs. On the cone arc the
uniform bounds, due to interval arithmetic,
\[
 0\le r\le\tfrac75,\qquad
 \tfrac9{10}\le\cos\theta\le1,\qquad
 -\tfrac13\le\sin\theta\le0
\]
give $0\le r\ip\mv{\gv(\theta)}\le1$ ,
using
\[
 \ip\mv{\gv(\theta)}
 =\nu\ip\mv{\eF}+\mu\cos\theta\ip\mv{\vF}
                      +\mu\sin\theta\ip\mv{\wF}.
\]
Since $\ip\mv D=-1/2$, the cone arc lies in the slab.
For the helices, use $\alpha\le1/4$, $\beta\le13/25$ and $m_z=0$ in
\[
 \ip\mv{H(\theta)}
 =-h\theta\ip\mv\eH-\tfrac12\cos\theta\ip\mv\vH,
 \qquad
 \ip\mv{H_c(\theta)}=m_xk_c\theta+m_y\sin\theta.
\]
Interval bounds $\left|\theta\right|\le1/4$, $\left|\cos\theta\right|\le1$ for the first
expression, and $\left|\theta\right|,\left|\sin\theta\right|\le13/25$ for the second, place
both projections in $[-1/2,1/2]$. Convexity covers every connecting
segment and the half-turn reverses the projection and covers the other arm.
Thus we get the directional width $W(\mv)=1$, which implies that the overall width, which is the minimum over all directions, can be no greater than one.

\begin{remark}
The curve touches this slab four alternating times, as in the argument by Wienholtz cited in~\cite{Ghomi} and its projection onto the plane whose normal is $\mv$ is the shortest for which we found these four alternating touches. The curves length in that projection is roughly $2.476$, significantly longer than the caliper curve's length of roughly $2.278$ that Ghomi uses in his lower bound on the length of a three-dimensional unit width curve in his Lemma 2.2 and Corollary 3.2. Replacing the caliper curve's length with the length of our projection yields the lower bound for our curve length $3.8896$, below its actual length of roughly $3.905$. Investigating these gaps may be of interest for further work towards proving the minimal length of a unit width curve.
\end{remark}

\subsection{Three analytic support estimates}\label{sec:analytic-support}
The following regions are homogeneous cones and all arguments follow for their negatives due to symmetry of directional width. Together with the reflected cone region they supply
four regions in the finite covering below.

\paragraph{Central region.}
Suppose
\begin{equation}\label{eq:centralregion}
 n_z\ge|n_x|,\qquad |n_y|\le\tan(0.65)n_z.
\end{equation}
Write $(n_y,n_z)=r(\sin\theta,\cos\theta)$, so $|\theta|\le0.65$.
The $yz$ projection of $H_c$ is a unit-circle arc. For $|\theta|\le\beta$,
$H_c(\theta)$ has scalar product $r$ with $(0,n_y,n_z)$ and
$|x|\le b$: explicitly,
$n_y\sin\theta+n_z\cos\theta=r$ and
$|k_c\theta|\le k_c\beta<b$, since $b=k_c(\beta+p)$ and $p>0$.
For $\beta\le\theta\le0.65$ use $C$. The first identity below follows
from~\eqref{eq:central-trig} and the inequality follows from interval arithmetic:
\[
 t\sin\beta+c\cos\beta=1,\qquad
 t\sin(0.65)+c\cos(0.65)\ge\tfrac{251}{250}>1.
\]
The positive function $t\sin\theta+c\cos\theta$ is concave on this
interval, so is at least one there. Negative $\theta$ uses $B$.
In both cases we have a point $P\in\mathcal K$ with
$n_yP_y+n_zP_z\ge r$ and $|P_x|\le b$. Choose the endpoint
$Q\in\{A,D\}$ with $n_xQ_x=-a|n_x|$. Then
\[
 W(\nv)\ge\ip\nv{P-Q}
 \ge r-b|n_x|+a|n_x|=r+(a-b)|n_x|.
\]
The region gives $r\ge n_z\ge|n_x|$, so the triangle
inequality in $\R^2$ yields
\[
 \norm\nv
 =\norm{(r-|n_x|)(1,0)+|n_x|(1,1)}
 \le r+(\sqrt2-1)|n_x|
 \le r+(a-b)|n_x|
 \le W(\nv),
\]
using $a-b>\sqrt2-1$ from~\eqref{eq:box}.

\paragraph{Paired-helix region.}
Let $q=-\ip \nv \vH$ and require
\begin{equation}\label{eq:pairregion}
 q\ge0,\qquad |\ip \nv \eH|\le0.15q,\qquad
 |n_z|\le\tan(0.25)q.
\end{equation}
For $\nv\ne0$, the region forces $q>0$. Put $x=\ip \nv \eH$,
$r=\sqrt{q^2+n_z^2}$, $\theta=\atan(n_z,q)$ and
$\delta=\arctan(x/r)$, the tilt of $\nv$ from the plane perpendicular
to the helix axis. Then $|\theta|\le0.25$ and
$|\delta|\le|x|/r\le0.15$. Interval arithmetic gives
\begin{equation}\label{eq:pairavailability}
 h\ge\tfrac12,\qquad
 0.4<\alpha+\min(\lambda/q_H,u).
\end{equation}
Evaluate $H$ at $\theta-\delta$ and its half-turn
image $H^*=R\circ H$ at $-\theta-\delta$.
The definitions $x=\ip\nv\eH$ and $q=-\ip\nv\vH$ give, for an angular
parameter $\vartheta$,
\begin{align*}
 \ip\nv{H(\vartheta)}&=z_0n_z-hx\vartheta+\tfrac r2\cos(\vartheta-\theta),\\
 \ip\nv{H^*(\vartheta)}&=z_0n_z+hx\vartheta-\tfrac r2\cos(\vartheta+\theta).
\end{align*}
The proposed difference trades separation perpendicular to the helix axis for a gain along
the helix axis:
\[
 \ip\nv{H(\theta-\delta)-H^*(-\theta-\delta)}
 =r\cos\delta+2hx\delta.
\]
If a parameter leaves the
helix interval, use its tangent-line continuation on the adjacent
segment. Both expressions have absolute value at most $0.4$,
so this point is permitted by~\eqref{eq:pairavailability}. The adjacent
segments extend the angular parameter by $\lambda/q_H$ and $u$.
For $\ip\nv H$ and $-\ip\nv{H^*}$, the circular terms are
concave between the actual and virtual parameters:
their cosine arguments have absolute value at most $0.4+0.25<\pi/2$.
A concave function lies below its tangent
line, so replacement by these segment points can only increase
the projected difference.
Since $\delta$ has the sign of $x$ and $|\sin\delta|\le|\delta|$,
we have $x\delta\ge x\sin\delta\ge0$. Thus $2h\ge1$ gives
\[
 W(\nv)\ge r\cos\delta+2hx\delta
 \ge r\cos\delta+x\sin\delta
 =\sqrt{r^2+x^2}=\norm\nv.
\]

\paragraph{Cone region.}
The developed endpoints obey $r_0>1$ and $r_1\cos(\mu\Theta)>1$,
as checked using interval arithmetic.
Thus the whole developed segment has first coordinate, and hence radius,
greater than one. Its continuous polar angle takes every value between
$0$ and $\mu\Theta$, so every $\theta\in[\Theta,0]$ occurs at a curve
point $P_F=D+r\gv(\theta)$ with $r>1$. Moving inward along its generator
stays in the convex hull:
\[
 D+\gv(\theta)=(1-r^{-1})D+r^{-1}P_F\in\mathcal K.
\]
Since $D,C\in\mathcal K$, the difference body contains $\overrightarrow{DC}$
and $C-(D+\gv(\theta))$, namely the circular arc
\begin{equation}\label{eq:section}
 S(\theta)=\overrightarrow{DC}-\gv(\theta)
 = (\lVert\overrightarrow{DC}\rVert-\nu)\eF-\mu(\cos\theta\,\vF+\sin\theta\,\wF),
 \qquad \Theta\le\theta\le0.
\end{equation}
For $Z=A,E$, define
$\lambda_Z=\nu/\ip{\eF}{Z-D}$ and
$Q_Z=\overrightarrow{DC}-\lambda_Z(Z-D)$. The interval bounds
$0<\lambda_Z<1$ give
$Q_Z=C-((1-\lambda_Z)D+\lambda_Z Z)\in\mathcal K-\mathcal K$.
The entry tangent plane contains $AD$, and the exit tangent plane
contains $DY$ and the tangent line $YE$ by~\eqref{eq:conetangent}, hence
also $DE$. The resulting identities
$\ip{\Nv(0)}{A-D}=\ip{\Nv(\Theta)}{E-D}=0$ give
\[
 Q_A=S(0)-\mu\tau_0\av_0,\qquad
 Q_E=S(\Theta)+\mu\tau_1\av_\Theta,
\]
Here $\lambda_Z(Z-D)$ has axial coordinate $\nu$, and the tangent-plane
identity forces its radial coordinate to be $\mu$. The remaining
component is along $\av_0$ or $\av_\Theta$, with
\[
 \tau_0=\frac{\lambda_A}{\mu}\ip{A-D}{\av_0},\qquad
 \tau_1=-\frac{\lambda_E}{\mu}\ip{E-D}{\av_\Theta}.
\]
Through interval arithmetic we verify $\tau_0,\tau_1>\tan(0.02)$.

For $0\le\delta\le0.04$,
\[
 \cos\delta+\tau_i\sin\delta-1
 =\sin\delta\bigl(\tau_i-\tan(\delta/2)\bigr)\ge0,
\]
because $\delta/2\le0.02$ and $\tau_i>\tan(0.02)$.
Therefore, for normals
\[
 \nv=x \eF-y(\cos\theta\,\vF+\sin\theta\,\wF),\quad y\ge0,
 \quad \Theta-0.04\le\theta\le0.04,
\]
the arc~\eqref{eq:section} and its two tangent continuations yield
\begin{equation}\label{eq:conebound}
 W(\nv)\ge\max\{\lVert\overrightarrow{DC}\rVert x,(\lVert\overrightarrow{DC}\rVert-\nu)x+\mu y\}.
\end{equation}
Here tangent continuation refers to the section points $Q_A,Q_E$.
For example, when $\theta=\delta\in[0,0.04]$,
\[
 \ip\nv{Q_A}=(\lVert\overrightarrow{DC}\rVert-\nu)x
      +\mu y(\cos\delta+\tau_0\sin\delta)
 \ge(\lVert\overrightarrow{DC}\rVert-\nu)x+\mu y.
\]
The point $Q_E$ gives the same estimate when $\theta=\Theta-\delta$.
Inside $[\Theta,0]$ use $S(\theta)$. The other term in
\eqref{eq:conebound} is $\ip\nv{\overrightarrow{DC}}=\lVert\overrightarrow{DC}\rVert x$.

For the finite cover we need linear inequalities keeping the normal in
this angular sector and ensuring $x\ge y\ge0$. Let
$\theta_-=\Theta-0.04$ and $\theta_+=0.04$ with the sector midpoint
$\Theta/2$. Let
\begin{equation}\label{eq:coneregion}
\begin{aligned}
 \ip{\sin\theta_-\,\vF-\cos\theta_-\,\wF}{\nv}&\ge0,\\
 \ip{-\sin\theta_+\,\vF+\cos\theta_+\,\wF}{\nv}&\ge0,\\
 \ip{\eF+1.1\bigl(\cos(\Theta/2)\vF+\sin(\Theta/2)\wF\bigr)}{\nv}&\ge0.
\end{aligned}
\end{equation}
For $y>0$, the first two left-hand sides are
$y\sin(\theta-\theta_-)$ and $y\sin(\theta_+-\theta)$, giving the
angular sector $\theta_-\le\theta\le\theta_+$.
The last inequality is $x-1.1y\cos(\theta-\Theta/2)\ge0$.
Since $|\theta-\Theta/2|\le0.04-\Theta/2\le1/4$,
\[
 x\ge1.1\cos(\theta-\Theta/2)\,y
 \ge\frac{11}{10}\left(1-\frac{(1/4)^2}{2}\right)y
 =\frac{341}{320}y\ge y\ge0.
\]
When $y=0$, the last inequality gives $x\ge0$.
Interval arithmetic yields $0\le d\le1$ on $\mathcal B$.

To compare~\eqref{eq:conebound} with $\norm\nv$, consider the planar points
\[
 U_F=(\lVert\overrightarrow{DC}\rVert,0),\qquad
 V_F=(\lVert\overrightarrow{DC}\rVert-\nu,\mu).
\]
Their joining line $X+dY=\lVert\overrightarrow{DC}\rVert=\sqrt{1+d^2}$ is tangent to the unit
circle because $d\mu=\nu$ from~\eqref{eq:cone_frame}. The point $U_F$ lies on the positive horizontal
axis, while $V_F$ has slope
$\mu/(\lVert\overrightarrow{DC}\rVert-\nu)=1/(1-d+d^2)\ge1$, since $0\le d\le1$.
Thus every ray through $(x,y)\ne(0,0)$ with $x\ge y\ge0$ meets
$[U_F,V_F]$ at a point $Z_F$ with $\norm{Z_F}\ge1$.
A linear projection on a segment is no larger than its largest endpoint
projection, so
\[
 \max\{\ip{(x,y)}{U_F},\ip{(x,y)}{V_F}\}
 \ge\ip{(x,y)}{Z_F}
 =\sqrt{x^2+y^2}\,\norm{Z_F}\ge\norm\nv.
\]
These endpoint projections are the two bounds in~\eqref{eq:conebound}.
At $\nv=0$ the inequality is an equality.
Half-turn symmetry supplies the reflected region.

\subsection{A finite covering of the remaining directions}\label{sec:finite-cover}
We cover the sphere of directions by spherical triangles, each certified
by one support pair or by containment in an analytic region. Write
$\operatorname{pos}(\nv_1,\nv_2,\nv_3)$ for the nonnegative span of its
vertex directions. The following lemma extends either check to the interior
(cf.\ \cite[Example~2.11]{BoydVandenberghe}).

\begin{lemma}\label{lem:cover}
If one pair $P,Q\in\mathcal K$ satisfies
$\ip{\nv_i}{P-Q}\ge\norm{\nv_i}$ for $i=1,2,3$, then
$W(\nv)\ge\norm\nv$ on $\operatorname{pos}(\nv_1,\nv_2,\nv_3)$.
If all three generators lie in an analytic region defined by homogeneous
linear inequalities, their nonnegative span lies in that region.
\end{lemma}
\begin{proof}
For $\nv=\sum_i\lambda_i\nv_i$, $\lambda_i\ge0$,
\[
 W(\nv)\ge\ip\nv{P-Q}
 \ge\sum_i\lambda_i\norm{\nv_i}
 \ge\norm{\nv}.
\]
The second assertion follows by linearity of each defining inequality.
\end{proof}

On the unit sphere, a fixed pair $P,Q$ certifies the spherical cap
$\{\nv\in S^2:\ip\nv{P-Q}\ge1\}$. Thus the support-pair part of
our verification is a spherical coverage problem, see
Petkovi\'c, Pokrajac and Latecki~\cite{SphericalCoverage}.
Exact circle arrangements provide another approach to coverage on
$S^2$~\cite{CazalsLoriot}. Here we combine the analytic regions above
with a triangular subdivision, certifying each cone by rational
interval inequalities. The general strategy of accepting certified
regions and subdividing undecided ones is standard in interval
methods~\cite{JaulinWalter}.

The supplementary material~\cite{bartels_2026_23309323} supplies a
covering of the sphere by 512 spherical triangles. Of these, 164 are checked
for containment in the analytic regions of Section~\ref{sec:analytic-support}.
Each of the other 348 is assigned one ordered pair $(P,Q)$ satisfying
$\ip{\nv_i}{P-Q}\ge\norm{\nv_i}$ at all three generators. These 348 triangles cover the complement
of the analytic regions and cover all directions together with it. The generating code is
\texttt{generate\_cone\_cover.py} in \texttt{open\_curve\_lean\_20261008} of the supplementary material~\cite{bartels_2026_23309323}. This is provided for exposition but we note, that it is not necessary to understand the correctness of the width proof. Verifying that the checked spherical triangles cover the complement of the already covered regions and that directional width is at least $1$ on each of them is sufficient.
Coverage of the whole sphere, including triangle boundaries, is
machine-verified in \path{OpenCurve/DirectionCover/Coverage.lean}
for every parameter vector in $\mathcal B$.

Let $\mathcal P=(P_i)_{i=0}^{32}$ be the list consisting of $A,B,C,D$,
five points at parameters $j/4$
($0\le j\le4$) on each of the four outer analytic arcs, and nine points
at parameters $j/8$ ($0\le j\le8$) on the central helix, each parameter
rescaled to $[0,1]$. Thus the list has $4+4\cdot5+9=33$ entries that lie exactly on the constructed curve. In the certificate, the
arcs are ordered as $F,H,H_c,H^*,F^*$ after the four endpoints.

We provide pseudo-code for the checking loop. A region label consists of
a region $r$ and a sign $\varepsilon\in\{-1,1\}$. Its halfspaces are
$\varepsilon\ip{\mathbf a_{r,j}}\nv\ge0$. The coefficient vectors
$\mathbf a_{r,j}$ come from~\eqref{eq:centralregion},
\eqref{eq:pairregion}, \eqref{eq:coneregion} and its half-turn image.
Each region has four stored halfspaces, for the cone regions the fourth
repeats the third.
Let $\underline{[f]}$ and $\overline{[f]}$ denote rational
interval bounds valid throughout $\mathcal B$. For a ray let
\[
 b(\nv)=\sum_{i=1}^3
 \max\bigl\{|\underline{[n_i]}|,|\overline{[n_i]}|\bigr\}^2
 \ge\norm\nv^2.
\]
\begin{quote}
\small
\textbf{Input:} the 512 cones and their support-pair or analytic-region labels.
\begin{tabbing}
\quad\=\quad\=\quad\=\quad\=\kill
\textbf{for} each final cone with generators $\nv_1,\nv_2,\nv_3$:\\
\>\textbf{if} its label is a support pair $(P,Q)$:\\
\>\>\textbf{for} $k=1,2,3$:\\
\>\>\>\textbf{if} $\nv_k=\pm\mv$: check the exact endpoint identities.\\
\>\>\>\textbf{else}: set $s=\underline{[\ip{\nv_k}{P-Q}]}$;\\
\>\>\>\>require $s\ge0$ and $s^2\ge b(\nv_k)$.\\
\>\textbf{else}, with region label $(r,\varepsilon)$:\\
\>\>\textbf{for} $k=1,2,3$ and $j=1,2,3,4$:\\
\>\>\>require $\underline{[\varepsilon\ip{\mathbf a_{r,j}}{\nv_k}]}\ge0$.
\end{tabbing}
\end{quote}
At $+\mv$, the support check requires $P\in\{A,C\}$ and
$Q\in\{B,D\}$. At $-\mv$ the sets are reversed. The endpoint
projections and $\norm\mv=1$ then give equality. Thus the 348 support
cones require $348\cdot3=1044$ generator checks: 16 exact equality
checks and 1028 interval checks, all of the latter with strict margins.
The analytic cones require $164\cdot3\cdot4=1968$ signed scalar-product
checks. These are counts of loop iterations, including repeated rays
and the repeated fourth halfspace slots.

\paragraph{Example: leaf 2.}
With $\mathbf e_x=(1,0,0)$, the third leaf in the programmatic check (numbered 2 with zero-based indexing) has ray generators  
\[
 \nv_1=-\tfrac14\mathbf e_x-\tfrac34\mathbf e_z,\qquad
 \nv_2=-\tfrac12\mathbf e_x-\tfrac12\mathbf e_z,\qquad
 \nv_3=-\tfrac14\mv-\tfrac34\mathbf e_z,
\]
and
\[
 P=D=(-a,0,0),\qquad
 Q=H_c(\beta/2)
   =\bigl(k_c\beta/2,\sin(\beta/2),\cos(\beta/2)\bigr).
\]
Write $\Delta=Q-P$. The spherical triangle has unit vertices
$\mathbf u_i=\nv_i/\norm{\nv_i}$. We check that the same ordered pair
$P,Q$ has projected separation at least one in each of these directions:
\[
 \ip{\mathbf u_i}{P-Q}
 =\frac{-\ip{\nv_i}{\Delta}}{\norm{\nv_i}}\ge1,
 \qquad i=1,2,3.
\]
The minus sign comes from $P-Q=-\Delta$, and the denominator accounts
for the unnormalized generators. A bound on $\norm\Delta$ alone would
not suffice: a segment longer than one can have zero projection in a
perpendicular direction.

Rounding the coordinate enclosures outward to six decimal places and using interval arithmetic gives
the following intervals, valid on $\mathcal B$:
\[
\begin{array}{c@{\quad}c@{\qquad}c@{\quad}c}
 a &[0.712497,\,0.712498]&m_x &[0.701756,\,0.701757]\\
 Q_x &[0.053346,\,0.053347]&m_y &[0.712416,\,0.712417]\\
 Q_y &[0.253313,\,0.253314]&Q_z &[0.967384,\,0.967385]
\end{array}
\]
Subtracting $P$ from $Q$ gives the coordinate enclosures
\[
\begin{aligned}
 \relax
 [\Delta_x]&=[0.765843,\,0.765845],\\
 [\Delta_y]&=[0.253313,\,0.253314],\\
 [\Delta_z]&=[0.967384,\,0.967385].
\end{aligned}
\]
All factors in the following expressions are positive, so using their
lower endpoints gives lower bounds $s_i$ for $-\ip{\nv_i}{\Delta}$:
\begin{align*}
 s_1&=\tfrac14(0.765843)+\tfrac34(0.967384)=0.91699875,\\
 s_2&=\tfrac12(0.765843)+\tfrac12(0.967384)=0.86661350,\\
 s_3&=\tfrac14\bigl(0.701756\cdot0.765843
                    +0.712416\cdot0.253313\bigr)
       +\tfrac34(0.967384)\\
    &=0.905012788629.
\end{align*}
Thus $\ip{\nv_i}{P-Q}\ge s_i>0$. The coordinatewise squared-norm
bounds used in the loop are $b(\nv_1)=5/8$, $b(\nv_2)=1/2$ and
\[
 b(\nv_3)=\frac{0.701757^2+0.712417^2}{16}+\frac9{16}
          =0.625000054308625.
\]
Exact rational arithmetic now gives
\[
\begin{array}{c@{\qquad}c}
 i&s_i^2-b(\nv_i)\\ \hline
 1&0.2158867075015625\\
 2&0.25101895838225\\
 3&0.194048093273414031699641
\end{array}
\]
Since $s_i>0$, the three positive squared margins imply
\[
 \ip{\nv_i}{P-Q}\ge s_i>\sqrt{b(\nv_i)}\ge\norm{\nv_i}.
\]
Thus the normalized projected separations exceed $1$ at all three
vertices. To see why this certifies the entire spherical triangle,
write any unit direction in it as
\[
 \mathbf u=\frac{\mathbf v}{\norm{\mathbf v}},\qquad
 \mathbf v=\sum_{i=1}^3\lambda_i\nv_i\ne0,\quad \lambda_i\ge0.
\]
The same pair $P,Q$ then gives
\begin{align*}
 W(\mathbf u)&\ge\ip{\mathbf u}{P-Q}
 =\frac{\sum_{i=1}^3\lambda_i\ip{\nv_i}{P-Q}}{\norm{\mathbf v}}\\
 &\ge\frac{\sum_{i=1}^3\lambda_i\norm{\nv_i}}
               {\norm{\sum_{i=1}^3\lambda_i\nv_i}}\ge1.
\end{align*}
Here the last inequality is the triangle inequality. This is
Lemma~\ref{lem:cover} applied to the three vertices of the spherical triangle. The interval bounds take the whole parameter box as inputs so the result intervals are guaranteed to contain the true values for the true parameters within that box. The supplied checker uses the tighter
original rational intervals. This leaf corresponds to \texttt{check2}
in the separately supplied Lean proof, part \texttt{OpenCurve/WidthCertificate/Checks000.lean}.

The supplied \texttt{python3 check\_cone\_witnesses.py} replays the
leaf inequalities using the rational bounds. Validity of these bounds
and all 512 leaf checks is proved in Lean, with the latter collected in
\path{OpenCurve/WidthCertificate/AllChecks.lean}.
Run \texttt{python3 verify\_all.py} in
\texttt{open\_curve\_lean\_20261008} to verify both coverage and the
width certificate.
Together with the containing slab this proves $w=1$ for every curve in
this family with $\xx\in\mathcal B$ and $\mathcal J(\xx)=0$.

\section{Reproducibility and the local conclusion}\label{sec:repro}
The figures were drawn with Matplotlib~\cite{Matplotlib}.
The supplementary material contains the search, an example 32-run population,
and the numerically proven width and length certificates. The programs
\path{vertex_search/population.py} and
\path{vertex_search/analyze_population.py} reproduce the search and
recheck its results, respectively.
The file \path{INTERVAL_INDEX.txt} maps the manuscript's interval bounds to the corresponding programs and Lean statements.

The Lean companion uses mathlib~\cite{Mathlib} and
LeanCert~\cite{LeanCert}. Its combined theorem is
\texttt{certified\_construction\_and\_family\_minimum} in
\path{OpenCurve/CompletedCertificate.lean}.
\texttt{verify\_all.py} checks the proofs and audits their axioms.
Only \texttt{propext}, \texttt{Classical.choice} and \texttt{Quot.sound}
are used.

The mpmath~\cite{Mpmath} certificates, which are independent from the Lean~4 proof, include the sharper
length enclosure~\eqref{eq:result} that we give in the above results, and are run by
\path{supplement/run_all.py}. Its width check
\path{supplement/width_symbolic/verify.py} uses a normal-fan certificate,
whereas Section~\ref{sec:finite-cover} uses the Lean cone subdivision. Both methods cover all directions.
The optional \path{supplement/sample_width.py} provides a direction-sampling
plausibility check.

Theorem~\ref{thm:main} establishes local minimality within the joining
family. We conjecture global minimality among all rectifiable curves. Our belief is supported by 1) the analytical construction being shorter than any numerically optimized polygonal chain among a large number of runs, 2) the 26 competitive runs described in Section~\ref{sec:search} converging to roughly the same shape and 3) the curve being proven locally minimal within a family describing that shape.

\begin{conjecture}\label{conj:global-minimum}
The curve $\Gamma$ is a shortest rectifiable curve in $\R^3$ of unit width, i.e.
every rectifiable curve $\gamma:[0,1]\to\R^3$ with $w(\gamma)=1$ satisfies
\[
 L(\gamma)\ge L(\Gamma).
\]
\end{conjecture}

The procedure in Section~\ref{sec:search} was also applied to closed
curves, a problem studied by Zalgaller~\cite{Zalgaller} and
Ghomi~\cite{Ghomi}. Our best certified polygon has $44{,}032$ vertices
and, after normalization to unit width, length at most $5.094393486530$.
A subsequent numerical fit improves the certified bound to $5.094393485661$. This is below the best ratio, roughly $5.1086$, that was found from replacing parts of the arcs on Ghomi's closed curve with straight line segments, a route for improvement that he already noted in his original paper. This suggests, that the optimal closed curve has a qualitatively different shape than Ghomi's curve with some arc pieces replaced by straight lines.
At the time of writing, we have no conclusive analytic hypothesis for
the minimizing closed curve. The supplementary material includes the
ordered polygon vertices and the certificates for both bounds.

\backmatter

\section*{Statements and Declarations}
\subsection*{Funding}
This work was supported by the 1+1+1 Joint Collaboration Fund of The Chinese University of Hong Kong, The Chinese University of Hong Kong, Shenzhen, and the Department of Science and Technology of Guangdong Province to Tongxin Li, Tinko Sebastian Bartels, and Ruixiang Wu under Grant No. 2025A0505000047.

\subsection*{Author contributions}
Tinko Sebastian Bartels: Implemented the refining optimization across polygonal chains, the optimization across the analytical curve family, the width plausibility check, the computer-assisted width, length and stationarity computations and proofs, and drafted the manuscript. 
Dennis Kevin Choy: Implemented the original refinement of the closed curve and the original width-testing program that identified a direction in which $L_3$ has width less than one, revised Section~\ref{sec:parameters}, and revised the figures.
Ruixiang Wu: Reviewed and revised miscellaneous parts of the manuscript and assisted with the implementation of the optimization across polygonal chains.
Tongxin Li: Reviewed and revised miscellaneous parts of the manuscript.
All authors approved the final manuscript and agree to be accountable
for all aspects of the work.

\subsection*{Competing interests}
The authors declare no competing interests.

\subsection*{Use of generative AI}
OpenAI's Codex with model ChatGPT 6 Astra assisted with coding for numerical experiments, optimization code, width verification code, figure generation code, as well as manuscript editing and checking of proofs, and the generation of the separately supplied Lean~4 proof formalization based on an earlier version of the arguments. The verification procedures and the scope of their guarantees are described in
Section~\ref{sec:repro}.

\subsection*{Data availability}
The source code, recorded parameter values, and verification certificates
are included in the companion materials described in Section~\ref{sec:repro} at~\cite{bartels_2026_23309323}.

\bibliography{references}

@article{AdhikariPitman,
  author = {Adhikari, Ani and Pitman, Jim},
  title = {The shortest planar arc of width 1},
  journal = {The American Mathematical Monthly},
  volume = {96},
  number = {4},
  pages = {309--327},
  year = {1989},
  doi = {10.1080/00029890.1989.11972188},
  url = {https://doi.org/10.1080/00029890.1989.11972188}
}

@article{Qhull,
  author = {Barber, C. Bradford and Dobkin, David P. and Huhdanpaa, Hannu},
  title = {The {Quickhull} algorithm for convex hulls},
  journal = {ACM Transactions on Mathematical Software},
  volume = {22},
  number = {4},
  pages = {469--483},
  year = {1996},
  doi = {10.1145/235815.235821},
  url = {https://doi.org/10.1145/235815.235821}
}

@article{Filters,
  author = {Bartels, Tinko and Fisikopoulos, Vissarion and Weiser, Martin},
  title = {Fast floating-point filters for robust predicates},
  journal = {BIT Numerical Mathematics},
  volume = {63},
  number = {2},
  pages = {31},
  year = {2023},
  doi = {10.1007/s10543-023-00975-x},
  url = {https://doi.org/10.1007/s10543-023-00975-x}
}

@inproceedings{Lean4,
  author = {de Moura, Leonardo and Ullrich, Sebastian},
  title = {The {Lean 4} theorem prover and programming language},
  editor = {Platzer, Andr{\'e} and Sutcliffe, Geoff},
  booktitle = {Automated Deduction---{CADE} 28},
  series = {Lecture Notes in Computer Science},
  volume = {12699},
  publisher = {Springer},
  address = {Cham},
  pages = {625--635},
  year = {2021},
  doi = {10.1007/978-3-030-79876-5_37},
}

@misc{DengGeneral,
  author = {Deng, Zhipeng},
  title = {A general solution to {Bellman}'s lost-in-a-forest problem},
  year = {2024},
  note = {arXiv:2412.10686},
  eprint = {2412.10686},
  archivePrefix = {arXiv},
  primaryClass = {math.OC},
  doi = {10.48550/arxiv.2412.10686},
  url = {https://arxiv.org/abs/2412.10686}
}

@misc{DengVariations,
  author = {Deng, Zhipeng},
  title = {Proof and more variations of {Bellman}'s lost-in-a-forest problem},
  year = {2026},
  note = {arXiv:2606.13987},
  eprint = {2606.13987},
  archivePrefix = {arXiv},
  primaryClass = {math.OC},
  doi = {10.48550/arxiv.2606.13987},
  url = {https://arxiv.org/abs/2606.13987}
}

@misc{DengStrip,
  author = {Deng, Zhipeng},
  title = {Revisit escape path for infinite unit strip forest and unit broadworm},
  year = {2026},
  note = {arXiv:2607.18563},
  eprint = {2607.18563},
  archivePrefix = {arXiv},
  primaryClass = {math.OC},
  doi = {10.48550/arxiv.2607.18563},
  url = {https://arxiv.org/abs/2607.18563}
}

@misc{BoostGeometry,
  author = {Gehrels, Barend and Lalande, Bruno and Loskot, Mateusz and Wulkiewicz, Adam and Karavelas, Menelaos and Fisikopoulos, Vissarion},
  title = {{Boost.Geometry}},
  note = {Version 1.92.0, documentation},
  url = {https://www.boost.org/doc/libs/1_92_0/libs/geometry/doc/html/index.html},
  urldate = {9 October 2026}
}

@article{Ghomi,
  author = {Ghomi, Mohammad},
  title = {The length, width, and inradius of space curves},
  journal = {Geometriae Dedicata},
  volume = {196},
  number = {1},
  pages = {123--143},
  year = {2018},
  doi = {10.1007/s10711-017-0312-3},
  url = {https://doi.org/10.1007/s10711-017-0312-3}
}

@misc{Clarabel,
  author = {Goulart, Paul J. and Chen, Yuwen},
  title = {{Clarabel}: An interior-point solver for conic programs with quadratic objectives},
  year = {2024},
  note = {arXiv:2405.12762},
  eprint = {2405.12762},
  archivePrefix = {arXiv},
  primaryClass = {math.OC},
  doi = {10.48550/arxiv.2405.12762},
  url = {https://arxiv.org/abs/2405.12762}
}

@article{NumPy,
  author = {Harris, Charles R. and Millman, K. Jarrod and {van der Walt}, St{\'e}fan J. and Gommers, Ralf and Virtanen, Pauli and Cournapeau, David and Wieser, Eric and Taylor, Julian and Berg, Sebastian and Smith, Nathaniel J. and Kern, Robert and Picus, Matti and Hoyer, Stephan and {van Kerkwijk}, Marten H. and Brett, Matthew and Haldane, Allan and {del R{\'i}o}, Jaime Fern{\'a}ndez and Wiebe, Mark and Peterson, Pearu and G{\'e}rard-Marchant, Pierre and Sheppard, Kevin and Reddy, Tyler and Weckesser, Warren and Abbasi, Hameer and Gohlke, Christoph and Oliphant, Travis E.},
  title = {Array programming with {NumPy}},
  journal = {Nature},
  volume = {585},
  number = {7825},
  pages = {357--362},
  year = {2020},
  doi = {10.1038/s41586-020-2649-2},
  url = {https://doi.org/10.1038/s41586-020-2649-2}
}

@article{HouleToussaint,
  author = {Houle, M. E. and Toussaint, G. T.},
  title = {Computing the width of a set},
  journal = {IEEE Transactions on Pattern Analysis and Machine Intelligence},
  volume = {10},
  number = {5},
  pages = {761--765},
  year = {1988},
  doi = {10.1109/34.6790},
  url = {https://doi.org/10.1109/34.6790}
}

@article{Matplotlib,
  author = {Hunter, John D.},
  title = {{Matplotlib}: A {2D} graphics environment},
  journal = {Computing in Science \& Engineering},
  volume = {9},
  number = {3},
  pages = {90--95},
  year = {2007},
  doi = {10.1109/mcse.2007.55},
  url = {https://doi.org/10.1109/mcse.2007.55}
}

@misc{Mpmath,
  author = {Johansson, Fredrik and {The mpmath developers}},
  title = {{mpmath}: A {Python} library for arbitrary-precision floating-point arithmetic},
  year = {2023},
  note = {Version 1.3.0},
  url = {https://mpmath.org/}
}

@misc{LeanCert,
  author = {{LeanCert Contributors}},
  title = {{LeanCert}: Verified interval arithmetic for {Lean 4}},
  year = {2026},
  note = {Software, commit \texttt{baf738991765}},
  url = {https://github.com/alerad/leancert/tree/baf738991765c9cd54665078244508d787ba012a}
}

@misc{BoostMultiprecision,
  author = {Maddock, John and Kormanyos, Christopher},
  title = {{Boost.Multiprecision}},
  note = {Version 1.92.0, documentation},
  url = {https://www.boost.org/library/1.92.0/multiprecision/},
  urldate = {9 October 2026}
}

@inproceedings{Mathlib,
  author = {{The mathlib Community}},
  title = {The {Lean} mathematical library},
  booktitle = {Proceedings of the 9th {ACM} {SIGPLAN} International Conference on Certified Programs and Proofs},
  series = {{CPP} 2020},
  publisher = {Association for Computing Machinery},
  address = {New York, NY, USA},
  pages = {367--381},
  year = {2020},
  doi = {10.1145/3372885.3373824},
}

@misc{Python,
  author = {{Python Software Foundation}},
  title = {{Python} 3.14 documentation},
  year = {2026},
  url = {https://docs.python.org/3.14/},
  urldate = {9 October 2026}
}

@article{SciPy,
  author = {Virtanen, Pauli and Gommers, Ralf and Oliphant, Travis E. and Haberland, Matt and Reddy, Tyler and Cournapeau, David and Burovski, Evgeni and Peterson, Pearu and Weckesser, Warren and Bright, Jonathan and {van der Walt}, St{\'e}fan J. and Brett, Matthew and Wilson, Joshua and Millman, K. Jarrod and Mayorov, Nikolay and Nelson, Andrew R. J. and Jones, Eric and Kern, Robert and Larson, Eric and Carey, C J and Polat, {\.I}lhan and Feng, Yu and Moore, Eric W. and VanderPlas, Jake and Laxalde, Denis and Perktold, Josef and Cimrman, Robert and Henriksen, Ian and Quintero, E. A. and Harris, Charles R. and Archibald, Anne M. and Ribeiro, Ant{\^o}nio H. and Pedregosa, Fabian and {van Mulbregt}, Paul and {SciPy 1.0 Contributors}},
  title = {{SciPy} 1.0: Fundamental algorithms for scientific computing in {Python}},
  journal = {Nature Methods},
  volume = {17},
  number = {3},
  pages = {261--272},
  year = {2020},
  doi = {10.1038/s41592-019-0686-2},
  url = {https://doi.org/10.1038/s41592-019-0686-2}
}

@article{ZalgallerPlanar,
  author = {Zalgaller, V. A.},
  title = {How to get out of the woods? ({On} a problem of {Bellman})},
  journal = {Matematicheskoe Prosveshchenie, Series 2},
  volume = {6},
  pages = {191--195},
  year = {1961},
  note = {In Russian. \url{https://www.mathnet.ru/eng/mp688}},
  url = {https://www.mathnet.ru/eng/mp688}
}

@article{Zalgaller,
  author = {Zalgaller, V. A.},
  title = {The problem of the shortest space curve of unit width},
  journal = {Matematicheskaya Fizika, Analiz, Geometriya},
  volume = {1},
  number = {3--4},
  pages = {454--461},
  year = {1994},
  note = {In Russian. \href{https://mag.ilt.kharkiv.ua/index.php/mag/article/view/m01-0454r}{Original article}; English translation posted by Steven Finch, arXiv:1910.02729 (2019), \url{https://arxiv.org/abs/1910.02729}},
  url = {https://mag.ilt.kharkiv.ua/index.php/mag/article/view/m01-0454r}
}

@Book{BoydVandenberghe,
  author    = {Boyd, Stephen and Vandenberghe, Lieven},
  publisher = {Cambridge University Press},
  title     = {Convex Optimization},
  year      = {2004},
  isbn      = {9780511804441},
  month     = Mar,
  doi       = {10.1017/cbo9780511804441},
  address   = {Cambridge}
}

@article{SphericalCoverage,
  author = {Petkovi{\'c}, Marko D. and Pokrajac, Dragoljub and Latecki, Longin Jan},
  title = {Spherical coverage verification},
  journal = {Applied Mathematics and Computation},
  volume = {218},
  number = {19},
  pages = {9699--9715},
  year = {2012},
  doi = {10.1016/j.amc.2012.03.014},
  url = {https://doi.org/10.1016/j.amc.2012.03.014}
}

@article{CazalsLoriot,
  author = {Cazals, Fr{\'e}d{\'e}ric and Loriot, S{\'e}bastien},
  title = {Computing the arrangement of circles on a sphere, with applications in structural biology},
  journal = {Computational Geometry},
  volume = {42},
  number = {6--7},
  pages = {551--565},
  year = {2009},
  doi = {10.1016/j.comgeo.2008.10.004},
  url = {https://doi.org/10.1016/j.comgeo.2008.10.004}
}

@article{JaulinWalter,
  author = {Jaulin, Luc and Walter, Eric},
  title = {Set inversion via interval analysis for nonlinear bounded-error estimation},
  journal = {Automatica},
  volume = {29},
  number = {4},
  pages = {1053--1064},
  year = {1993},
  doi = {10.1016/0005-1098(93)90106-4},
  url = {https://doi.org/10.1016/0005-1098(93)90106-4}
}

@Article{bellman1956minimizationproblem,
  author    = {Richard Bellman},
  journal   = {Bulletin of the American Mathematical Society},
  title     = {Minimization Problem},
  year      = {1956},
  issn      = {0002-9904},
  month     = May,
  number    = {3},
  pages     = {270--271},
  volume    = {62},
  doi       = {10.1090/s0002-9904-1956-10021-9},
  publisher = {American Mathematical Society (AMS)},
}

@article{Finch01102004,
author = {Steven R. Finch and John E. Wetzel},
title = {Lost in a Forest},
journal = {The American Mathematical Monthly},
volume = {111},
number = {8},
pages = {645--654},
year = {2004},
publisher = {Taylor \& Francis},
doi = {10.1080/00029890.2004.11920126},
}

@misc{bartels_2026_23309323,
  author       = {Bartels, Tinko Sebastian and
                  Choy, Dennis Kevin and
                  Wu, Ruixiang and
                  Li, Tongxin},
  title        = {Supplementary Material to "A short open space
                   curve of unit width"
                  },
  month        = oct,
  year         = 2026,
  publisher    = {Zenodo},
  doi          = {10.5281/zenodo.23309323},
}
\end{document}